\documentclass[12pt,draft]{article}

\usepackage{amsmath,amssymb,amsthm,xspace,amscd}

\theoremstyle{remark}
\newtheorem{remark}{Remark}

\theoremstyle{plain}
\newtheorem{theorem}{Theorem}
\newtheorem{statement}{Proposition}

\newcommand{\be}{{\mathrm b}}
\newcommand{\en}{{\mathrm e}}

\newcommand{\A}{{\mathcal A}}
\newcommand{\Inv}{{\mathcal I}}
\newcommand{\R}{{\mathbb R}}

\newcommand{\vect}{{\mathbf v}}
\newcommand{\wect}{{\mathbf w}}
\newcommand{\xect}{{\mathbf x}}
\newcommand{\yect}{{\mathbf y}}
\newcommand{\kect}{{\mathbf k}}
\newcommand{\cect}{{\mathbf c}}
\newcommand{\Aect}{{\mathbf A}}

\begin{document}

\title{Interval Rearrangement Ensembles}
\author{Alexey TEPLINSKY
\thanks{Institute of Mathematics, Natl.\ Acad.\ Sc.\ Ukraine; e-mail:\ teplinsky.a@gmail.com}}
\maketitle

\begin{abstract}
We introduce a new concept of interval rearrangement ensembles (IRE), which is a generalization of interval exchange transformations (IET). This construction expands the space of IETs in accordance with the natural duality that we pinpoint. Induction of Rauzy\,--\,Veech kind is applicable to IREs. It is conjugate to the reverse operation by the duality mentioned above. A natural extension of an IRE is associated with two transversal flows on a flat translation surface with branching points.
\end{abstract}

\section{Introduction}\ \
Interval exchange transformations (IETs) are one-dimensional dynamical systems, which are widely studied since the 1970s starting with works~\cite{Keane75, Veech78, Rauzy79, Keane-Rauzy80, Masur82, Veech82}. A simple and graphic presentation of the subject can be found, for ex., in~\cite{Viana06}.

But our motivation for generalizing this construction comes from recent achievements in investigation of circle diffeomorphisms with breaks~\cite{TeplinskyKhanin04rus, Teplinsky08ukr, KhaninTeplinsky13, CunhaSmania13, CunhaSmania14}, which can be considered as a special case of nonlinear IETs. The renormalization approach leads to studying dynamics of action of the renormalization operator on a limiting finite-dimensional area of parameters. One of the key observations, which made the results of paper~\cite{KhaninTeplinsky13} possible, was its authors' discovery of an important symmetry (represented in the form of a simple involution on the limiting set mentioned above, see~\cite{Teplinsky08ukr} for the details) that conjugates the renormalization operator to its reverse. This time reversibility allowed the authors to prove hyperbolicity of renormalization. They called this symmetry ``duality''.

While the case of a single break studied in~\cite{KhaninTeplinsky13} led to consideration of an (non-linear) exchange of only two intervals, which is an IET with a very simple discrete component, the case of multiple breaks required understanding of how do much more complex IETs on multiple segments transform under renormalization procedures. We managed to extend the duality found for the simplest case to the space of all rotational IETs (this is a class of multi-segment IETs induced by irrational circle rotations). This duality reverses induction on descrete data, on linear natural extensions, and on linear-fractional IETs.

In the case of a rotational IET, the dual system is a rotational IET as well. But what will happen if one applies the operation of duality to non-rotational IET? The answer to this question appeared somehow unexpected: the result is not a IET, and not even a dynamical system, but rather a more general object that we named ``interval rearrangement ensemble'' (IRE). The space of all such ensembles is closed w.r.t.\ our duality and appears to be very natural in many senses, which we explain in the present paper.

Let us describe briefly the structure of our paper. In Secs.~2 and~3 we present classical IETs (on a single segment) and IETs on multiple segments, and in Sec.~4 we give the definition of an IRE. Then we explain what are the dynamical systems corresponding to IREs (Sec.~5), calculate dimensions of their real components (Sec.~6), define the elementary steps of induction (Sec.~7), and the time-reversing duality (Sec.~8). Natural extensions for IREs are given in Sec.~\ref{sect:extension}, and the corresponding translation surfaces and flows are presented in Sec..~\ref{sect:surfaces} (which also contains the detailed analysis of one classical example). In Sec.~\ref{sect:conclusions} we sum up the obtained results.

\section{Classical IETs}
\label{sect:single}

Below is the classical definition of an IET.

Let us be given an \emph{alphabet} $\A$ of $d\ge1$ symbols, two discrete bijective mappings $\pi_\be$ and $\pi_\en,$ acting from $\A$ onto the set $\{1, \ldots, d\},$ and a \emph{vector of lengths} $\vect = (v_\alpha)_{\alpha\in\A}\in\R^\A_ + .$ The \emph{interval exchange transformation} with this data $(\pi_\be, \pi_\en, \vect)$ is a one-dimensional dynamical system on the half-open segment $J = \left[0, \sum_{\alpha\in\A}v_\alpha\right),$ which is given by the discontinuous bijective map
\begin{equation}
\label{eq:IET}
f: x\mapsto x-x_\be(\alpha) + x_\en(\alpha)
\quad
\text{for}
\quad
x\in I_{\alpha\be}, \alpha\in\A,
\end{equation}
where $I_{\alpha\be} = [x_\be(\alpha), x_\be(\alpha) + v_{\alpha}),$ $x_\be(\alpha) = \sum_{\beta\colon \pi_\be(\beta)<\pi_\be(\alpha)}v_{\beta},$ $x_\en(\alpha) = \sum_{\beta\colon \pi_\en(\beta)<\pi_\en(\alpha)}v_{\beta}.$ The subscripts ``$\be$'' and  ``$\en$'' here come from the words ``beginning'' and ``ending'' respectively.

It is easy to check that this map splits the segment $J$ into $d$ \emph{beginning intervals} $I_{\alpha\be}$
(here and in what follows we use the terms ``segment'' and ``interval'' non-interchangeably, calling by them objects of logically different matter) of lengths $v_\alpha,$ $\alpha\in\A,$ ordered in accordance with $\pi_\be$ (meaning that the beginning interval $I_{\alpha\be}$ takes the position number $\pi_\be(\alpha)$ counting from left to right), and shifts each of them separately onto the corresponding \emph{ending intervals} $I_{\alpha\en} = [x_\en(\alpha), x_\en(\alpha) + v_{\alpha})$ of the same lengths, but now ordered in accordance with $\pi_\en.$ 

It is important to note that the data set $(\pi_\be, \pi_\en, \vect),$ which determines $(J, f),$ is a hybrid object consisting of a discrete component (or \emph{scheme}) $(\pi_\be, \pi_\en)$ and a real component~$\vect.$

The scheme of an IET is often being written in so-called two-row (or two-line) notation
\begin{equation}
\label{eq:two-rows-notation}
\begin{pmatrix}
\pi_\be^{-1}(1) & \pi_\be^{-1}(2) & \ldots & \pi_\be^{-1}(d)
\\[\medskipamount]
\pi_\en^{-1}(1) & \pi_\en^{-1}(2) & \ldots & \pi_\en^{-1}(d)
\end{pmatrix}.
\end{equation}
For ex., $\begin{pmatrix}
\alpha & \beta & \gamma & \delta
\\
\delta & \gamma & \beta & \alpha
\end{pmatrix}$ denotes the scheme of an IET that rearranges four intervals in the opposite order.

Note, that if one considers an IET as a single particular dynamical system, then the scheme above has no difference from, say, the scheme $\begin{pmatrix}
\beta & \alpha & \gamma & \delta
\\
\delta & \gamma & \alpha & \beta
\end{pmatrix}.$ 
But for those manipulations with IETs, which we describe below, it is useful to follow the symbolic labels of the intervals, and it is worthy to consider the mathematical object of IET not just as a one-dimensional dynamical system $(J, f),$, but namely as the data set $(\pi_\be, \pi_\en, \vect),$ which is an element of certain parameter space, on which meta-dynamical systems act.

One of the classical meta-systems acting on the space of IETs is the Rauzy\,--\,Veech induction.
It transforms the original IET $(\pi_\be, \pi_\en, \vect)$ into the new one $(\pi'_\be, \pi'_\en, \vect'),$ which is, in the dynamical sense, a map of the first return of trajectories onto the reduced segment determined by comparing lengths of the rightmost beginning interval and the rightmost ending interval, and cutting away the shorter one of them from the original segment $J.$ In formulas: if $v_{\pi_\be^{-1}(d)}<v_{\pi_\en^{-1}(d)},$ then the segment of new IET is $J' = \left[0, \sum_{\alpha\in\A\backslash\{\pi_\be^{-1}(d)\}}v_\alpha\right);$ if the opposite inequality holds, then $J' = \left[0, \sum_{\alpha\in\A\backslash\{\pi_\en^{-1}(d)\}}v_\alpha\right).$ One may easily calculate the first return map onto this segment.

In the case when $v_{\pi_\be^{-1}(d)}<v_{\pi_\en^{-1}(d)}$:
\begin{enumerate}
\item[] 
$\pi'_\be(\alpha) = \pi_\be(\alpha)$ for all $\alpha$ such that $\pi_\be(\alpha)\le\pi_\be(\pi_\en^{-1}(d)),$

\item[] 
$\pi'_\be(\alpha) = \pi_\be(\alpha) + 1$ for all $\alpha$ such that $\pi_\be(\pi_\en^{-1}(d))<\pi_\be(\alpha)<d,$ 

\item[] 
$\pi'_\be(\pi_\be^{-1}(d)) = \pi_\be(\pi_\en^{-1}(d)) + 1,$

\item[] 
$\pi'_\en(\alpha) = \pi_\en(\alpha)$ for all $\alpha\in\A,$

\item[] 
$v'_{\pi_\en^{-1}(d)} = v_{\pi_\en^{-1}(d)}-v_{\pi_\be^{-1}(d)},$

\item[] 
$v'_{\alpha} = v_\alpha$ for all $\alpha\ne\pi_\en^{-1}(d).$
\end{enumerate}

In the case when $v_{\pi_\be^{-1}(d)}>v_{\pi_\en^{-1}(d)}$:
\begin{enumerate}
\item[] 
$\pi'_\be(\alpha) = \pi_\be(\alpha)$ for all $\alpha\in\A,$

\item[] 
$\pi'_\en(\alpha) = \pi_\en(\alpha)$ for all $\alpha$ such that $\pi_\en(\alpha)\le\pi_\en(\pi_\be^{-1}(d)),$

\item[] 
$\pi'_\en(\alpha) = \pi_\en(\alpha) + 1$ for all $\alpha$ such that $\pi_\en(\pi_\be^{-1}(d))<\pi_\en(\alpha)<d,$

\item[] 
$\pi'_\en(\pi_\en^{-1}(d)) = \pi_\en(\pi_\be^{-1}(d)) + 1,$

\item[] 
$v'_{\pi_\be^{-1}(d)} = v_{\pi_\be^{-1}(d)}-v_{\pi_\en^{-1}(d)},$

\item[] 
$v'_{\alpha} = v_\alpha$ for all $\alpha\ne\pi_\be^{-1}(d).$
\end{enumerate}

This transformation of an IET is easier to comprehend in two-row notation (\ref{eq:two-rows-notation}): if the interval with label $\pi_\be^{-1}(d)$ is strictly shorter than the interval with label $\pi_\en^{-1}(d),$ then in the upper row the symbol $\pi_\be^{-1}(d)$ is removed from its rightmost position and inserted into the same row straight to the right of $\pi_\en^{-1}(d),$ and the lower row does not change; if the interval with label $\pi_\be^{-1}(d)$ is strictly longer than the interval with label $\pi_\en^{-1}(d),$ then the upper row does not change, while in the lower row the symbol $\pi_\en^{-1}(d)$ is moved from its rightmost position into the position straight to the right of $\pi_\be^{-1}(d)$ in the same row. After this change in the discrete component, the length of the longest one among those two intervals decreases by the value of the length of the shorter one.

\begin{remark}
The transformation described above is not defined in the case $v_{\pi_\be^{-1}(d)} = v_{\pi_\en^{-1}(d)}.$ When applying Rauzy\,--\,Veech induction it is assumed that such a case never happens. Usually this assumption is guaranteed by imposing the Keane condition~\cite{Keane75} formulated as follows: the trajectories of all break points $x_\be(\alpha),$ $\alpha\in\A\backslash\{\pi_\be^{-1}(1)\},$ are mutually disjoint. The Keane condition also implies minimality of the IET $(J, f)$: each trajectory of this dynamical system is everywhere dense on the segment $J.$
\end{remark}

\section{Multi-segment IETs}
\label{sect:multi}
Classical IETs, which were defined in the previous section, act on a single segment, and the first step in their generalization is removing this restriction, i.e., considering IETs on multiple segments. We also get rid of pinning down some edge at the point~0. In what follows, we will call multi-segment IETs just IETs.

Consider an alphabet $\A$ of $d\ge1$ symbols and two bijective mappings $\pi_\be$ and $\pi_\en,$ acting from $\A$ onto $\{1, \ldots, d\},$ like we did before. But now add to this data two ordered sets of $N\ge1$ positive integers $\kect_\be = (k_{1\be}, \ldots, k_{N\be})$ and $\kect_\en = (k_{1\en}, \ldots, k_{N\en}),$ each of them summing up to $d,$ i.e,  $ \sum_{s = 1}^{N}k_{s\be} = \sum_{s = 1}^{N}k_{s\en} = d.$
Let us have a vector of lengths $\vect = (v_\alpha)_{\alpha\in\A}\in\R_ + ^\A$ satisfying
$$
\sum_{i = 1}^{k_{s\be}}v_{\pi_\be^{-1}(k_{1\be} + \ldots + k_{(s-1)\be} + i)} = \sum_{i = 1}^{k_{s\en}}v_{\pi_\en^{-1}(k_{1\en} + \ldots + k_{(s-1)\en} + i)},
\quad
1\le s\le N.
$$
Denote the sums above as $\ell_s,$ $1\le s\le N.$ It is easy to see that if any $N-1$ of the equalities above are true, then they are all true, because summing up all these $N$ equalities gives an identity with the expression $\sum_{\alpha\in\A}v_\alpha$ in both sides.

For a given set of segments $J_s = [A_s, B_s)$ of lengths $B_s-A_s = \ell_s,$ $1\le s\le N,$ an IET (on multiple segments) with the data $(\pi_\be, \pi_\en, \kect_\be, \kect_\en, \vect)$ is defined on the disjoint union $\bigsqcup_{s = 1}^NJ_s$ as a map that splits the $N$ segments $J_s$ into $d$ beginning intervals $I_{\alpha\be}$ of lengths $v_\alpha,$ $\alpha\in\A,$ ordered in accordance with $\pi_\be$ and $\kect_\be$ (meaning that the interval $I_{\alpha\be}$ takes the position number $\pi_\be(\alpha)-(k_{1\be} + \ldots + k_{(s-1)\be}),$ counting from left to right, on the segment $J_s,$ $1\le s\le N$), and shifts each of them separately onto the corresponding ending intervals $I_{\alpha\en}$ of the same lengths $v_\alpha,$ $\alpha\in\A,$ ordered in accordance with  $\pi_\en$ and $\kect_\en.$ Formally, this map is given by the same expression (\ref{eq:IET}), but now the left endpoint of the beginning segment $I_{\alpha\be}$ is given as
$x_\be(\alpha)  = A_{s_\be(\alpha)} + \sum_{\beta\colon k_{1\be} + \ldots + k_{({s_\be(\alpha)}-1)\be}<\pi_\be(\beta)<\pi_\be(\alpha)}v_{\beta},$ where $s_\be(\alpha) = \min\{s\colon \pi_\be(\alpha) \le k_{1\be} + \ldots + k_{s\be}\}$ is the number of the segment containing this beginning interval, while the left endpoint of the ending interval $I_{\alpha\en}$ is given as $x_\en(\alpha) = A_{s_\en(\alpha)} + \sum_{\beta\colon k_{1\en} + \ldots + k_{({s_\en(\alpha)}-1)\en}<\pi_\en(\beta)<\pi_\en(\alpha)}v_{\beta},$ where $s_\en(\alpha) = \min\{s\colon \pi_\en(\alpha) \le k_{1\en} + \ldots + k_{s\en}\}$ is the number of the segment containing this ending interval.

I is easy to see that for $N = 1$ and $A_1 = 0$ this definition is equivalent to the definition of a classical IET from the previous section. It is also worth mentioning that in the case $N\ne1$ the discrete component of the IET data $(\pi_\be, \pi_\en, \kect_\be, \kect_\en)$ is what elevates the complexity of the whole construction to the next level.

This scheme is naturally visualized with the two-row notation as follows:
\begin{gather}
\left(
\begin{bmatrix}
\pi_\be^{-1}(1) & \ldots & \pi_\be^{-1}(k_{1\be})
\\[\medskipamount]
\pi_\en^{-1}(1) & \ldots & \pi_\en^{-1}(k_{1\en}) \end{bmatrix}
\begin{bmatrix}
\pi_\be^{-1}(k_{1\be} + 1) & \ldots & \pi_\be^{-1}(k_{1\be} + k_{2\be})
\\[\medskipamount]
\pi_\en^{-1}(k_{1\en} + 1) & \ldots & \pi_\en^{-1}(k_{1\en} + k_{2\en}) \end{bmatrix} 
\ldots \right.
\notag \\
\left.\ldots
\begin{bmatrix}
\pi_\be^{-1}(k_{1\be} + \ldots + k_{(N-1)\be} + 1) & \ldots & \pi_\be^{-1}(d)
\\[\medskipamount]
\pi_\en^{-1}(k_{1\en} + \ldots + k_{(N-1)\en} + 1) & \ldots & \pi_\en^{-1}(d) \end{bmatrix}\right).
\label{eq:two-row-notation-multi}
\end{gather}
In this form of notation, $N$ blocks in square brackets correspond to splitting of segments $J_s,$ $1\le s\le N,$ into beginning (upper row) and ending (lower row) intervals with labels written there (from left to right, the same as the intervals themselves are positioned).

The Rauzy\,--\,Veech induction also must be generalized for the new construction. There is no sense anymore to keep the left endpoints of segments untouched, so it is only natural to allow cutting intervals from both sides, not from the right only---unlike how it happens in the classical Rauzy\,--\,Veech induction, described in the previous section. Presently we will define elementary induction steps of four types: ``cropping a beginning interval on the right'', ``cropping an ending interval on the right'', ``cropping a beginning interval on the left'', and ``cropping the ending interval on the left'', denoting these operations as $\Pi_{\alpha\beta}^{\mathrm{rb}},$ $\Pi_{\alpha\beta}^{\mathrm{re}},$ $\Pi_{\alpha\beta}^{\mathrm{lb}}$, and $\Pi_{\alpha\beta}^{\mathrm{le}}$ respectively. The superscripts ``r'' and ``l'' here come from the words ``right'' and ``left'', while $\alpha,$ $\beta,$ $\alpha\ne\beta,$ are labels of the beginning and ending intervals engaged in the operation, respectively.

The elementary induction steps $\Pi_{\alpha\beta}^{\mathrm{rb}}$ and $\Pi_{\alpha\beta}^{\mathrm{re}}$ are applicable only if for some $1\le s^*\le N$ the equalities $s_\be(\alpha) = s_\en(\beta) = s^*,$ $\pi_\be(\alpha) = k_{1\be} + \ldots + k_{s^*\be}$ and $\pi_\en(\beta)  = k_{1\en} + \ldots + k_{s^*\en}$ hold, i.e., both the beginning interval $I_{\alpha\be}$ and the ending interval $I_{\beta\en}$ are adjacent to the right endpoint of the segment $J_{s^*}.$ Similarly to the classical Rauzy\,--\,Veech induction, we crop the longest one among these two intervals by the length of the shorter one. Hence, in the case $v_\alpha>v_\beta,$ one can apply the operation $\Pi_{\alpha\beta}^{\mathrm{rb}},$ while in the case $v_\alpha<v_\beta$ the operation $\Pi_{\alpha\beta}^{\mathrm{re}}$ is applicable. For $v_\alpha = v_\beta$ these operations are not defined.

The steps $\Pi_{\alpha\beta}^{\mathrm{lb}}$ and $\Pi_{\alpha\beta}^{\mathrm{le}}$ are applicable only if $s_\en(\alpha) = s_\be(\beta) = s^*,$ $\pi_\be(\alpha) = k_{1\be} + \ldots + k_{(s^*-1)\be} + 1,$ $\pi_\en(\beta) = k_{1\en} + \ldots + k_{(s^*-1)\en} + 1$ for some $1\le s^*\le N,$ i.e., both the beginning $I_{\alpha\be}$ and the ending $I_{\beta\en}$ intervals are adjacent to the left endpoint of the segment $J_{s^*}.$ In the case $v_\alpha>v_\beta,$ one can apply $\Pi_{\alpha\beta}^{\mathrm{lb}},$ while in the case $v_\alpha<v_\beta,$ the operation $\Pi_{\alpha\beta}^{\mathrm{lb}}$ is applicable. For the case $v_\alpha = v_\beta$ the operations are not defined.

Let us denote the discrete data of an IET before applying an induction step as $(\pi_\be, \pi_\en, \kect_\be, \kect_\en),$ and after its application as $(\pi'_\be, \pi'_\en, \kect'_\be, \kect'_\en).$ We define the action of the elementary induction steps as follows (in each of the four cases we describe this action by words first, having in mind the two-row notation for the scheme, and then formally).

$\Pi_{\alpha\beta}^{\mathrm{rb}}$ moves the symbol $\beta$ from its rightmost position in the lower row of the $s^*$th block into a position straight to the right of the symbol $\alpha$ in the lower row of a block that contains it:
\begin{enumerate}
\item[]
$\pi'_\be = \pi_\be,$ $\kect'_\be = \kect_\be,$

\item[]
$\pi'_\en(\gamma) = \pi_\en(\gamma)$ for $\gamma$ such that either $\pi_\en(\gamma)<\pi_\en(\beta)$ and $\pi_\en(\gamma)\le\pi_\en(\alpha),$ or $\pi_\en(\gamma)>\pi_\en(\beta)$ and $\pi_\en(\gamma)>\pi_\en(\alpha),$

\item[]
$\pi'_\en(\gamma) = \pi_\en(\gamma) + 1$ for $\gamma$ such that $\pi_\en(\alpha)<\pi_\en(\gamma)<\pi_\en(\beta),$

\item[]
$\pi'_\en(\gamma) = \pi_\en(\gamma)-1$ for  $\gamma$ such that $\pi_\en(\beta)<\pi_\en(\gamma)\le\pi_\en(\alpha),$

\item[]
$\pi'_\en(\beta) = \pi_\en(\alpha) + 1$ if $\pi_\en(\alpha)<\pi_\en(\beta),$ otherwise $\pi'_\en(\beta) = \pi_\en(\alpha),$

\item[]
$k'_{s\en} = k_{s\en}$ for all $s\not\in\{s^*, s_\en(\alpha)\}$ or $s = s^* = s_\en(\alpha),$

\item[]
$k'_{s^*\en} = k'_{s^*\en}-1$ if $s^*\ne s_\en(\alpha),$

\item[]
$k'_{s_\en(\alpha)\en} = k'_{s_\en(\alpha)\en} + 1$ if $s^*\ne s_\en(\alpha).$
\end{enumerate}

$\Pi_{\alpha\beta}^{\mathrm{re}}$ moves the symbol $\alpha$ from its rightmost position in the upper row of the $s^*$th block into a position straight to the right of the symbol $\beta$ in the upper row of a block that contains it:
\begin{enumerate}
\item[]
$\pi'_\en = \pi_\en,$ $\kect'_\en = \kect_\en,$

\item[]
$\pi'_\be(\gamma) = \pi_\be(\gamma)$ for $\gamma$ such that either $\pi_\be(\gamma)<\pi_\be(\alpha)$ and $\pi_\be(\gamma)\le\pi_\be(\beta),$ or $\pi_\be(\gamma)>\pi_\be(\alpha)$ and $\pi_\be(\gamma)>\pi_\be(\beta),$

\item[]
$\pi'_\be(\gamma) = \pi_\be(\gamma) + 1$ for $\gamma$ such that $\pi_\be(\beta)<\pi_\be(\gamma)<\pi_\be(\alpha),$

\item[]
$\pi'_\be(\gamma) = \pi_\be(\gamma)-1$ for $\gamma$ such that $\pi_\be(\alpha)<\pi_\be(\gamma)\le\pi_\be(\beta),$

\item[]
$\pi'_\be(\alpha) = \pi_\be(\beta) + 1$ if $\pi_\be(\beta)<\pi_\be(\alpha),$ otherwise $\pi'_\be(\alpha) = \pi_\be(\beta),$

\item[]
$k'_{s\be} = k_{s\be}$ for all $s\not\in\{s^*, s_\be(\beta)\}$ or $s = s^* = s_\be(\beta),$

\item[]
$k'_{s^*\be} = k'_{s^*\be}-1$ if $s^*\ne s_\be(\beta),$

\item[]
$k'_{s_\be(\beta)\be} = k'_{s_\be(\beta)\be} + 1$ if $s^*\ne s_\be(\beta).$
\end{enumerate}

$\Pi_{\alpha\beta}^{\mathrm{lb}}$ moves $\beta$ from its leftmost position in the lower row of the $s^*$th block into a position straight to the left of $\alpha$ in the lower row of a block that contains it:
\begin{enumerate}
\item[]
$\pi'_\be = \pi_\be,$ $\kect'_\be = \kect_\be,$

\item[]
$\pi'_\en(\gamma) = \pi_\en(\gamma)$ for $\gamma$ such that either $\pi_\en(\gamma)<\pi_\en(\alpha)$ and $\pi_\en(\gamma)<\pi_\en(\beta),$ or $\pi_\en(\gamma)\ge\pi_\en(\alpha)$ and $\pi_\en(\gamma)>\pi_\en(\beta),$

\item[]
$\pi'_\en(\gamma) = \pi_\en(\gamma) + 1$ for $\gamma$ such that $\pi_\en(\alpha)\le\pi_\en(\gamma)<\pi_\en(\beta),$

\item[]
$\pi'_\en(\gamma) = \pi_\en(\gamma)-1$ for $\gamma$ such that $\pi_\en(\beta)<\pi_\en(\gamma)<\pi_\en(\alpha),$

\item[]
$\pi'_\en(\beta) = \pi_\en(\alpha)-1$ if $\pi_\en(\beta)<\pi_\en(\alpha),$ otherwise $\pi'_\en(\beta) = \pi_\en(\alpha),$

\item[]
$k'_{s\en} = k_{s\en}$ for all $s\not\in\{s^*, s_\en(\alpha)\}$ or $s = s^* = s_\en(\alpha),$

\item[]
$k'_{s^*\en} = k'_{s^*\en}-1$ if $s^*\ne s_\en(\alpha),$

\item[]
$k'_{s_\en(\alpha)\en} = k'_{s_\en(\alpha)\en} + 1$ if $s^*\ne s_\en(\alpha).$
\end{enumerate}

$\Pi_{\alpha\beta}^{\mathrm{le}}$ moves $\alpha$ from its leftmost position in the upper row of the $s^*$th block into a position straight to the left of $\beta$ in the upper row of a block that contains it:

\begin{enumerate}
\item[]
$\pi'_\en = \pi_\en,$ $\kect'_\en = \kect_\en,$

\item[]
$\pi'_\be(\gamma) = \pi_\be(\gamma)$ for $\gamma$ such that either $\pi_\be(\gamma)<\pi_\be(\beta)$ and $\pi_\be(\gamma)<\pi_\be(\alpha),$ or $\pi_\be(\gamma)\ge\pi_\be(\beta)$ and $\pi_\be(\gamma)>\pi_\be(\alpha),$

\item[]
$\pi'_\be(\gamma) = \pi_\be(\gamma) + 1$ for $\gamma$ such that $\pi_\be(\beta)\le\pi_\be(\gamma)<\pi_\be(\alpha),$

\item[]
$\pi'_\be(\gamma) = \pi_\be(\gamma)-1$ for $\gamma$ such that $\pi_\be(\alpha)<\pi_\be(\gamma)<\pi_\be(\beta),$

\item[]
$\pi'_\be(\alpha) = \pi_\be(\beta)-1$ if $\pi_\be(\alpha)<\pi_\be(\beta),$ otherwise $\pi'_\be(\alpha) = \pi_\be(\beta),$

\item[]
$k'_{s\be} = k_{s\be}$ for all $s\not\in\{s^*, s_\be(\beta)\}$ or $s = s^* = s_\be(\beta),$

\item[]
$k'_{s^*\be} = k'_{s^*\be}-1$ if $s^*\ne s_\be(\beta),$

\item[]
$k'_{s_\be(\beta)\be} = k'_{s_\be(\beta)\be} + 1$ if $s^*\ne s_\be(\beta).$
\end{enumerate}

It must be noted that for each of the four steps there is a possibility that $(\pi_\be, \pi_\en, \kect_\be, \kect_\en)  = (\pi'_\be, \pi'_\en, \kect'_\be, \kect'_\en),$ i.e., the discrete data is left totally unchanged (for the step $\Pi_{\alpha\beta}^{\mathrm{rb}}$ this happens when the symbol $\alpha$ in the lower row of the $s^*$th block is already located straight to the left of $\beta;$ similarly for other steps).

The real component of an IET may be defined as a combination of the vector of lengths $\vect \in\R_ + ^\A$ with the vector of the left endpoints of segments $\Aect = (A_1, \ldots, A_N)\in\R^N.$ It contains $d + N$ real numbers in total; they are not independent, however, and we will research the question of dimensions of corresponding spaces in Sec.~\ref{sect:dimensions}.

The changes in the real data from $(\vect, \Aect)$ to $(\vect', \Aect')$ due to the action of four elementary induction steps are written much simpler than the changes in the discrete data, namely:
\begin{enumerate}
\item[]
for $\Pi_{\alpha\beta}^{\mathrm{rb}}$ and $\Pi_{\alpha\beta}^{\mathrm{lb}}$ we have $v'_\alpha = v_\alpha-v_\beta;$ while for $\Pi_{\alpha\beta}^{\mathrm{re}}$ and $\Pi_{\alpha\beta}^{\mathrm{le}}$ we have $v'_\beta = v_\beta-v_\alpha,$ and all other lengths do not change;

\item[]
for $\Pi_{\alpha\beta}^{\mathrm{lb}}$ we have $A'_{s^*} = A_{s^*} + v_\beta,$ and for $\Pi_{\alpha\beta}^{\mathrm{le}}$ we have $A'_{s^*} = A_{s^*} + v_\alpha,$ and the left endpoints of all other segments but $J_{s^*}$ do not change; while for ``cropping on the right'' steps $\Pi_{\alpha\beta}^{\mathrm{rb}}$ and $\Pi_{\alpha\beta}^{\mathrm{re}}$ the left endpoints of all segments $J_s,$ $1\le s\le N,$ are left unchanged.
\end{enumerate}

Let us remark that the elementary induction steps act on a discrete component of IET data $(\pi_\be, \pi_\en, \kect_\be, \kect_\en)$ independently on the value of its real component $(\vect, \Aect).$

\section{Definition of an IRE}
\label{sect:definition}
In the following sections, the motivation for generalized construction that we describe below will become more clear, but now let us look at the discrete component of IET data from a different viewpoint.

To begin with, let us transform the two-row notation (\ref{eq:two-row-notation-multi}) into a one-row notation by arming the symbols of our alphabet from the upper and lower rows with letters  ``b'' and ``e'' respectively and writing them down as $N$ finite sequences, one for each segment $J_s$: first we write the symbols of the upper row from left to right, and after them (i.e., to the right of them) the symbols of the lower row, but from right to left. For ex., instead of $\begin{pmatrix}
\alpha & \beta &
\\
\gamma & \delta & \varepsilon
\end{pmatrix}$ we write $(\alpha\be, \beta\be, \varepsilon\en, \delta\en, \gamma\en).$

Accordingly to this form of notation, we code each segment $J_s,$ $1\le s\le N,$ which is a union of beginning intervals with labels (from left to right) $\pi_\be^{-1}(k_{1\be} + \ldots + k_{(s-1)\be} + 1),  \ldots, \pi_\be^{-1}(k_{1\be} + \ldots + k_{s\be})$ and at the same time it is a union of ending intervals with labels \linebreak $\pi_\en^{-1}(k_{1\en} + \ldots + k_{(s-1)\en} + 1), \ldots, \pi_\en^{-1}(k_{1\en} + \ldots + k_{s\en}),$ as a single {\em cycle}
\begin{gather*}
\Big({\pi_\be^{-1}(k_{1\be} + \ldots + k_{(s-1)\be} + 1)}\be, \ldots, {\pi_\be^{-1}(k_{1\be} + \ldots + k_{s\be})}\be,
\\
{\pi_\en^{-1}(k_{1\en} + \ldots + k_{s\en})}\en, \ldots, {\pi_\en^{-1}(k_{1\en} + \ldots + k_{(s-1)\en} + 1)}\en\Big)
\end{gather*}
in the \emph{doubled alphabet} $\bar\A = \A\times\{\be, \en\}.$ (In the sequel we denote elements of $\bar\A$ as $\bar\xi = \alpha\mathrm{m}  = (\alpha, \mathrm{m}),$ where $\alpha\in\A,$ $\mathrm{m}\in\{\be, \en\}.$) We call these finite sequences cycles, meaning cycles in a permutation of the set $\bar\A.$ Effectively, this is a one-to-one relation of following, so that every such a sequence is closed into a cycle, i.e., its first element follows the last one.

It is clear, that in our construction for an IET each element of the doubled alphabet $\bar\A$ belongs exactly to one cycle, and therefore the scheme (i.e., discrete data) of an IET can be defined as a permutation of $\bar\A$ such that each one of its cycles can be split into two arcs, the first arc containing only elements from $\A\times\{\be\},$ the second arc---only from $\A\times\{\en\}.$

Now we eliminate the latter restriction and consider any permutation $\sigma$ of the set $\bar\A$ as a \emph{IRE scheme}, where ``IRE'' stands for interval rearrangement ensemble, the new mathematical object that we introduce.

Formally, an IRE scheme $\sigma$ can be equivalently viewed either a bijective mapping of $\bar\A$ onto itself, or a $(2d\times2d)$ matrix $\bar\A^2\to\{0, 1\},$ in which every row and every column contains exactly one unity, or a corresponding oriented graph with $2d$ vertices $\bar\A,$ which is a union of mutually disjoint cycles. When $\sigma(\bar\xi) = \bar\eta,$ we say that $\bar\xi$ stands just (or right) before $\bar\eta,$ and that $\bar\eta$ stands right (or just) after $\bar\xi,$ $\bar\xi, \bar\eta\in\bar\A.$

Interval rearrangement ensemble is a pair $(\sigma, \xect),$ in which $\sigma\colon \bar\A\to\bar\A$ is a discrete bijective mapping (permutation{),} and $\xect\in\R^{\bar\A}$ is a vector, whose coordinates satisfy the equalities
\begin{equation}
\label{eq:endpoints_relation}
x_{\alpha\be} + x_{\alpha\en}-x_{\sigma(\alpha\be)} + x_{\sigma(\alpha\en)} = 0
\quad
{\text{for all}}
\quad
\alpha\in\A.
\end{equation}

It is clear that for every given permutation $\sigma$ the set of all \emph{vectors of endpoints} $\xect,$ which satisfy the relations (\ref{eq:endpoints_relation}) (we call such vectors \emph{allowed} by the scheme $\sigma$), forms a linear space $X_\sigma\subset\R^{\bar\A}.$

The \emph{vector of lengths} $\vect\in\R^{\A}$ for a given IRE $(\sigma, \xect)$ is defined coordinate-wise as
\begin{equation}
\label{eq:lengths_relation}
v_\alpha = x_{\sigma(\alpha\be)}-x_{\alpha\be} = x_{\alpha\en}-x_{\sigma(\alpha\en)},
\quad
\alpha\in\A,
\end{equation}
accordingly to (\ref{eq:endpoints_relation}). The set of all vectors $\vect,$ \emph{allowed} by the scheme $\sigma$ (i.e., obtained by substituting coordinates of an allowed vector of endpoints $\xect$ into the relations (\ref{eq:lengths_relation}), forms a linear apace $V_\sigma\subset\R^{\A}.$

We call a pair $(\sigma, \vect),$ with a vector of lengths $\vect\in V_\sigma,$ a \emph{floating} IRE, meaning by this that different segments (which correspond to cycles in the scheme $\sigma$) can be freely shifted along the coordinate axis independently on one another, unlike for a \emph{pinned down} IRE $(\sigma, \xect),$ where each segment (cycle) occupies some fixed position on the coordinate axis.

\begin{remark}
The relations (\ref{eq:lengths_relation}) can be rewritten in the form $x_{\sigma(\bar\xi)} = x_{\bar\xi}\pm v_\alpha$ for every $\bar\xi = \alpha\mathrm{m}\in\bar\A,$ where the plus sign corresponds to the case of $\mathrm{m} = \be,$ and the minus sign---to the case of $\mathrm{m} = \en.$ This form of relations allows to obtain the vector $\xect$ from a known vector $\vect$ and just a single endpoint value $x_{\bar\xi}$ for each cycle in the permutation $\sigma;$ the rest of the endpoint values in that cycle are calculated by the formulas
\begin{equation}
\label{eq:lengths-to-endpoints}
x_{\sigma^i(\bar\xi)} = x_{\bar\xi}\pm v_{\alpha_0}\pm\ldots\pm v_{\alpha_{i-1}},
\end{equation}
where $\alpha_j\in\A$ is the first component of $\sigma^j(\bar\xi)\in\bar\A,$ and the sign in front of $v_{\alpha_j}$ is determined by the second component of $\sigma^j(\bar\xi),$ $0\le j<i.$
\end{remark}

We call an IRE \emph{positive}, if $v_\alpha>0$ for all $\alpha\in\A.$ An IRE scheme is called \emph{positive}, if it allows a positive IRE.

A positive IRE should be visualized as an ensemble of $2d$ intervals, which are coupled in $d$ pairs with labels $\alpha  \in\A.$ In each such pair, the beginning interval $I_{\alpha\be} = [x_{\alpha\be}, x_{\sigma(\alpha\be)})$ and the ending interval $I_{\alpha\en} = [x_{\sigma(\alpha\en)}, x_{\alpha\en})$ with the same label $\alpha$ have the same length $v_\alpha.$ All these intervals are connected by their endpoints accordingly to the cycles in the scheme thus forming one or several closed one-dimensional polygonal chains (which we call by the same word ``cycles''). In the case of IET these polygonal chains cover the segments $J_s,$ $1\le s\le N,$ passing them twice: along the beginning and along the ending intervals.

\section{IREs as dynamical systems}
\label{sect:dynsys}
In this section we show, which dynamical systems are associated with IREs.

For an IRE scheme $\sigma$, an instance of $\sigma(\alpha\be) = \beta\en,$ $\alpha, \beta\in\A,$ is called a \emph{turn back} at $\beta\en,$ and an instance of $\sigma(\beta\en) = \alpha\be,$ $\alpha, \beta\in\A,$ is called a \emph{turn forward} at $\alpha\be.$ It is obvious that every cycle in the permutation $\sigma$ contains an equal number ob turns back and turns forward. We call a cycle \emph{twisted}, if it contains more than one turn back. Define the \emph{twist nummber} of a cycle as the number of turns back in it minus one. The twist number $T = T(\sigma)$ of an IRE scheme $\sigma$ is the sum of the twist numbers of all its cycles; in other words, it is the total number of turns back in this scheme minus the number $N = N(\sigma)$ of its cycles.

\begin{remark}
In a degenerate case, when a cycle consists either of beginning intervals only, or of ending intervals only, its twist number is~$-1$. Any IRE scheme containing such a cycle is not positive.
\end{remark}

If no cycle in a scheme of a positive IRE is twisted, then this IRE is a IET (it is a classical IET, if its scheme consists of a sungle cycle, and the onle turn forward happens at a symbol $\alpha\be$ such that $x_{\alpha\be} = 0$). An IRE of such type determines a dynamical system described in Sec.~\ref{sect:multi}: the cycles correspond to segments, which are split, on the one hand, into beginning intervals, and on the other hand, into ending intervals, and the bijection $f$ maps the former ones onto the latter ones.

However, if a positive IRE is not a IET, i.e., its scheme contains at least one twisted cycle, then some of the beginning intervals in that cycle necessarily overlap, as well as do some of the ending intervals in this cycle (every point of the real axis is covered with the same number of beginning and ending intervals belonging to the same cycle). This configuration does not determine a dynamical system uniquely. The most natural approach here is to associate an IRE with a family of dynamical systems, which are effectively
\emph{IETs on trees}. In order to obtain a tree from every particular cycle of a positive IRE, one has to fix a set of special points numbered by the twist number of this cycle. The overlapping intervals will be joined at these points and disjoint beyond them. The resulting phase space will be a tree with disjoint components corresponding to the cycles in the IRE (some authors use to call a tree with disjoint components a ``forest'', but we abstain). The tree obtained will be the union of all beginning intervals, as well as the union of all ending intervals, branching at the special points. The dynamical system on this tree will be determined by an ``almost bijective'' discontinuous map that shifts every beginning interval onto the corresponding ending interval, but the branching points have several (two, in a generic case) images and preimages.

\subsection{Transforming an IRE into an IET on a tree.}
\label{subsect:algorithm}

Here we present one of possible algorithms that transforms a positive IRE into a dynamical system, described informally in the previous paragraph as an IET on a tree. For an untwisted cycle, its beginning and ending segments can be seen as glued first to second, because each point of the corresponding segment is covered by exactly one beginning and one ending interval. If the cycle is twisted, then there are whole subsegments covered by more than one beginning interval (and by the same number of ending intervals), and we have to determine artificially how exactly the beginning and ending intervals will be glued first to second at each point. To do this, we are going to use the following inductive procedure, which must be applied to each twisted cycle separately.

Start by considering a closed polygonal chain $(x_{\alpha_1\be}, x_{\beta_1\en}, x_{\alpha_2\be}, x_{\beta_2\en}, \ldots, x_{\alpha_{t + 1}\be}, x_{\beta_{t + 1}\en}, x_{\alpha_1\be})$ with vertices at consecutive turns in our cycle (turn forward at $\alpha_1\be,$ then turn back at $\beta_1\en,$ then forward at $\alpha_2\be,$ back at $\beta_2\en$ etc.). Here $t$ is the twist number of the cycle. The segments of this polygonal chain consist alternately of beginning and ending points (i.e., the points belonging to beginning and ending intervals). We find the shortest one of these $2(t + 1)$ segments. Assume for definiteness that the shortest segment is a segment of beginning points $[x_{\alpha_i\be}, x_{\beta_i\en}]$ (for a segment of ending points $[x_{\beta_i\en}, x_{\alpha_{i + 1}\be}]$ the procedure is similar). Choose an arbitrary point $c_1$ on this shortest segment and declare that each point belonging to one of the beginning intervals with labels between $\alpha_i\be$ and $\beta_i\en$ in our cycle ($\beta_i\en$ is not included, and the rest are beginning since there are no other turns between the vertices of this polygonal line in the cycle) and lies to the left of $c_1,$ is glued to the corresponding point with the same coordinate, which belongs to one of the ending intervals with labels between $\beta_{i-1}\en$ and $\alpha_i\be$ in this cycle. On the other hand, each point belonging to one of the beginning intervals with labels between $\alpha_i\be$ and $\beta_i\en$ in our cycle and lies to the right of $c_1,$ is glued to the corresponding point with the same coordinate, which belongs to one of the ending intervals with labels between $\beta_{i}\en$ and $\alpha_{i + 1}\be.$ The beginning point with coordinate $c_1$ itself is glued to both corresponding ending points; it will be a branching point of the tree. (Here $\beta_{0}\en = \beta_{t + 1}\en,$ $\alpha_{t + 2}\be = \alpha_{1}\be,$ since the whole construction is cyclic.) Now, having determined the destination of beginning and ending points belonging to this part of our poligonal line, we exclude them from our consideration and return $c_1$ as an ending point (glued of two), reducing our task to a shorter poligonal chain $(x_{\alpha_1\be}, x_{\beta_1\en}, \ldots, x_{\alpha_{i-1}\be}, x_{\beta_{i-1}\en}, x_{\alpha_{i + 1}\be}, x_{\beta_{i + 1}\en}, \ldots, x_{\alpha_{t + 1}\be}, x_{\beta_{t + 1}\en}, x_{\alpha_1\be}),$ in which the segments of beginning and ending points alternate as before, but now it consists of $2t$ segments only. We again find the shortest of its segments, choose a point $c_2$ on it and do a similar gluing and obtain a polygonal chain of $2(t-1)$ segments. Continue with this procedure until a polygonal chain of just two segments $(x_{\alpha_m\be}, x_{\beta_n\en}, x_{\alpha_m\be})$ is left. Finally, in it we glue every beginning point to the ending point with the same coordinate, thus ending the process. The branching points $c_1, \ldots, c_t$ are determined, and the cycle became a tree.

After applying the procedure described above to each of the cycles in $\sigma,$ we obtain a dynamical system (with a finite number of trajectories that branch), which is effectively an interval exchange on a tree. The total number of branching points equals to the twist number $T(\sigma).$ An important feature of this algorithm is that every beginning interval is split into a final number of subintervals, and each one of them is glued as a whole to some ending interval, and vice versa. A detailed example of such gluing can be found in Sec.\,\ref{sect:surfaces}.

\begin{remark}
\label{remark:question}
A smart question would be: why studying an IET on a tree, when it can be transformed into an IET on multiple segments by cutting all the brunches at branching points and making separate segments of them?
The answer will become more clear after we describe the induction, the duality, the natural extension and corresponding flows for an IRE in sections that follow. Putting it shortly, an IRE has to be considered as aone-dimensional Poincare section for a certain two-dimensional continuous time system, and cutting branches on such a section cannot simplify the extended picture: is a surface contains special points, they cannot be just removed, but can be only shifted from (or to) a specific section.\
\end{remark}

We will explain this in more detail in subsequent sections; now let us investigate the spaces
$X_\sigma$ and $V_\sigma.$

\section{Dimensions of real data spaces}
\label{sect:dimensions}
An IRE scheme $\sigma$ with an alphabet $\A$ is called \emph{reducible}, if there exists a subset $\A_0\subset\A,$ $\A_0\not\in\{\varnothing, \A\}$ such that $\sigma(\bar\A_0) = \bar\A_0$ \big(here $\bar\A_0 = \A_0\times\{\be, \en\}$\big). The restriction of $\sigma$ to $\bar\A_0$ is an IRE scheme on its own, as is the restriction of $\sigma$ to $\bar\A\,\backslash\bar\A_0,$ so any IRE with this scheme $\sigma$ can be reduced to two independent IREs with smaller alphabets $\A_0$ and $\A\backslash\A_0.$ If such subset does not exist, then we call a scheme \emph{irreducible}. It is clear that for a given scheme $\sigma$ its alphabet $\A$ uniquely (up to an order) splits into subalphabets $\A_i,$ $1\le i\le P$ such that $\bigcup_{i = 1}^P\A_i = \A,$ and each restriction $\sigma_i = \sigma|_{\bar\A_i},$ $1\le i\le P,$ is irreducible. It is natural to call $\{\sigma_i\}_{1\le i\le P}$ the \emph{partition of $\sigma$ into irreducible components}. With a little abuse of terminology, we will refer to the subalphabets $\A_i,$ $1\le i\le P,$ themselves as to irreducible components (of the alphabet). It is obvious that the number $P = P(\sigma)$ of irreducible components of the scheme $\sigma$ falls in the range $1\le P\le d.$

\begin{statement}
\label{prop:X_dimension}
$\dim(X_\sigma) = d + P.$
\end{statement}

\begin{proof}
Write down the system of linear equations (\ref{eq:endpoints_relation}) in its matrix form
\begin{gather*}
\Delta\xect = 0
\end{gather*}
and notice some properties of thus defined $d\times2d$ matrix $\Delta$, the dimension of whose kernel we need to calculate. It is easy to see that every its column either contains one element~$1,$ one element~$-1,$ and all the rest are~$0,$ or consists entirely of elements~$0.$ Every its row $\Delta_\alpha,$ $\alpha\in\A,$ either contains two elements~$1,$ two elements~$-1,$ and all the rest are~$0,$ or one element~$1,$ one element~$-1,$ and all the rest are~$0,$ or else consists entirely of elements~$0.$ The sum of its rows $\sum_{\alpha\in\A_0}\Delta_\alpha$ taken over any subset $\A_0\subset\A$ contains equal number of elements~1 and~$-1,$ and the rest of its elements are~0; $\sum_{\alpha\in\A_0}\Delta_\alpha = 0$ if and only if $\sigma(\bar\A_0) = \bar\A_0,$ i.e., when $\A_0$ is a union of irreducible components. This property, in particular, implies that the cokernel of the matrix $\Delta$ contains vectors $\cect^{(i)}\in\R^\A$ with components $c^{(i)}_\alpha = 1$ for $\alpha\in\A_i$ and~$0$ for the rest, where $\A_i,$ $1\le i\le P,$ are the irreducible components. We need to prove that these $P$ linearly independent vectors span the whole cokernel of the matrix $\Delta.$

Let us assume that $\sum_{\alpha\in\A}c_\alpha\Delta_\alpha = 0$ with certain coefficients $\cect\in\R^\A$ and $c_{\alpha^{(1)}} = c\ne0$ for some $\alpha^{(1)}\in\A.$ If $\Delta_{\alpha^{(1)}} = 0,$ then $\{\alpha^{(1)}\}$ is an irreducible component. Otherwise, $\Delta_{\alpha^{(1)}}$ contains an element~$1,$ and there exists unique $\alpha^{(2)}\in\A$ such that the row $\Delta_{\alpha^{(2)}}$ contains an element~$-1$ in the same position. Therefore, $c_{\alpha^{(2)}} = c.$ If the sum $\Delta_{\alpha^{(1)}} + \Delta_{\alpha^{(2)}} = 0,$ then $\{\alpha^{(1)}, \alpha^{(2)}\}$ is a union of irreducible components. Otherwise, $\Delta_{\alpha^{(1)}} + \Delta_{\alpha^{(2)}}$ contains an element~$1,$ and there exists unique $\alpha^{(3)}\in\A$ such that $\Delta_{\alpha^{(3)}}$ contains an element~$-1$ in the same position. Therefore, $c_{\alpha^{(3)}} = c.$ We continue like this until the process stops. The stopping point is determining a union of irreducible components $\A_0$ such that $c_\alpha = c$ for all $\alpha\in\A_0.$ Since $\sum_{\alpha\in\A_0}\Delta_\alpha = 0,$ we remove these addends from the sum we started with and obtain the sum $\sum_{\alpha\in\A\backslash\A_0}c_\alpha\Delta_\alpha = 0$ over a smaller set of symbols. If there is still a non-zero coefficient left in it, then by applying the same procedure we find and remove another union of irreducible components. Continuing like this, we will inevitably stop with the conclusion that the vector of coefficients $\cect$ is indeed a linear combination of the vectors $\cect^{(i)},$ $1\le i\le P.$

Thus we have explicitly described the cokernel of the matrix $\Delta$ and shown that its dimension equals~$P.$ The rank of this matrix is therefore $d-P$, and the dimension of its kernel is $d + P.$

Proposition~\ref{prop:X_dimension} is proved.
\end{proof}

We have denoted by $N = N(\sigma)$ the number of cycles in the permutation $\sigma.$ A bound $N\ge P$ takes place, since each irreducible component $\sigma_i,$ $1\le i\le P,$ is a permutation, and therefore contains at least one cycle.

\begin{statement}
\label{prop:V_dimension}
$\dim(V_\sigma) = d + P-N.$
\end{statement}

\begin{proof}
The linear space $V_\sigma$ is an image of the linear space $X_\sigma$ under the action of the linear operator determined by the equalities (\ref{eq:lengths_relation}). We will explicitly describe the kernel of this linear operator. The permutation $\sigma$ consists of $N$ cycles, and the doubled alphabet $\bar\A$ splits into corresponding sets $S_i,$ $1\le i\le N.$ For every $\bar\xi\in S_i$ we have $S_i = \bigcup_{j = 0}^\infty\sigma^j(\bar\xi).$ The relations~(\ref{eq:lengths_relation}) imply that $\vect = 0$ if and only if $x_{\sigma(\bar\xi)} = x_{\bar\xi}$ for all $\bar\xi\in\bar\A.$ Threfore, the kernel we are describing consists of all vectors $\sum_{i = 1}^Nc_i\xect^{(i)},$ where $x^{(i)}_{\bar\xi} = 1$ for $\bar\xi\in S_i,$ and the rest of the elements are~0, $c_i\in\R,$ $1\le i\le N.$ Since the vectors $\xect^{(i)},$ $1\le i\le N,$ belong to $X_\sigma,$ are linearly independent and span the whole kernel, its dimension equals~$N.$ Now it follows from Proposition~\ref{prop:X_dimension} that the dimension of $V_\sigma$ is $d + P-N.$

Proposition~\ref{prop:V_dimension} is proved.
\end{proof}

Note, that this proof implies the non-obvious inequality $N\le d + P.$

It also justifies the concept of floating IREs $(\sigma, \vect)$ as the result of factorization of the space of pinned down IREs $(\sigma, \xect)$ w.r.t.\ free shifting of their cycles along the coordinate axis (though the number of cycles is different for different schemes).

\section{Induction for IREs}
\label{sect:induction}
We define four elementary induction steps $\Pi_{\alpha\beta}^{\mathrm{rb}},$ $\Pi_{\alpha\beta}^{\mathrm{re}},$ $\Pi_{\alpha\beta}^{\mathrm{lb}},$ and $\Pi_{\alpha\beta}^{\mathrm{le}}$ in the framework of the IRE concept. They directly generalize the steps we defined in Sec.~\ref{sect:multi}. Therefore, we reserve the same notations for them, although now these operators will be acting on different data, both discrete and real. And not going into unnecessary formalities, we will denote and call these steps in the same way whether they act on IRE schemes $\sigma$ (they act on schemes independently on real data), on IRE proper, both pinned down $(\sigma, \xect)$ and floating $(\sigma, \vect),$ or even on real components of IREs separately (when a scheme is fixed by the context). In Sec.~\ref{sect:extension} we will define natural extensions of IREs, and the four steps will be acting on those as well.

The steps $\Pi_{\alpha\beta}^{\mathrm{rb}}$ and $\Pi_{\alpha\beta}^{\mathrm{re}}$ are applicable in the case of $\sigma(\alpha\be) = \beta\en$ (turn back at $\beta\en,$ see Sec.~\ref{sect:dynsys}), while the steps $\Pi_{\alpha\beta}^{\mathrm{lb}}$ and $\Pi_{\alpha\beta}^{\mathrm{le}}$ are applicable in the case of $\sigma(\beta\en) = \alpha\be$ (turn bak at $\alpha\be$).

In terms of cycles in the permutation $\sigma$, the elementary induction steps act in the following way: the step $\Pi_{\alpha\beta}^{\mathrm{rb}}$ \big(the steps $\Pi_{\alpha\beta}^{\mathrm{re}},$ $\Pi_{\alpha\beta}^{\mathrm{lb}},$ $\Pi_{\alpha\beta}^{\mathrm{le}}$\big) move the element $\beta\en$ (the elements $\alpha\be,$ $\beta\en,$ $\alpha\be$) from its current position into a position straight before $\alpha\en$ (straight after $\beta\be,$ straight after $\alpha\en,$ straight before $\beta\be$). It is possible that the starting and the finishing position of the moved elements is the same (for ex., when the step $\Pi_{\alpha\beta}^{\mathrm{rb}},$ is used, and $\beta\en$ is already standing straight before $\alpha\en$), in such a situation the scheme does not change. (By saying that $\bar\xi$ stands straight before $\bar\eta,$ and $\bar\eta$ stand straight after $\bar\xi,$ we mean that $\sigma(\bar\xi) = \bar\eta,$ $\bar\xi, \bar\eta\in\bar\A.$)

Formally, let $\sigma'$ be an IRE scheme obtained from the scheme $\sigma$ as a result of applying an induction step. The action of the four steps on schemes is written as follows:
\begin{description}
\item[$\Pi_{\alpha\beta}^{\mathrm{rb}}$:]
if $\sigma(\beta\en) = \alpha\en,$ then $\sigma' = \sigma;$
otherwise $\sigma'(\alpha\be) = \sigma(\beta\en),$ $\sigma'(\beta\en) = \alpha\en,$ $\sigma'(\sigma^{-1}(\alpha\en)) = \beta\en,$
$\sigma'(\bar\xi) = \sigma(\bar\xi)$ for all $\bar\xi\in\bar\A\,\backslash\{\alpha\be, \beta\en, \sigma^{-1}(\alpha\en)\},$

\item[$\Pi_{\alpha\beta}^{\mathrm{re}}$:]
if $\sigma(\beta\be) = \alpha\be,$ then $\sigma' = \sigma;$
otherwise $\sigma'(\sigma^{-1}(\alpha\be)) = \beta\en,$ $\sigma'(\beta\be) = \alpha\be,$ \linebreak$\sigma'(\alpha\be)  = \sigma(\beta\be),$
$\sigma'(\bar\xi) = \sigma(\bar\xi)$ for all $\bar\xi\in\bar\A\,\backslash\{\sigma^{-1}(\alpha\be), \beta\be, \alpha\be\},$

\item[$\Pi_{\alpha\beta}^{\mathrm{lb}}$:]
if $\sigma(\alpha\en) = \beta\en,$ then $\sigma' = \sigma;$
otherwise $\sigma'(\sigma^{-1}(\beta\en)) = \alpha\be,$ $\sigma'(\alpha\en) = \beta\en,$ $\sigma'(\beta\en) = \sigma(\alpha\en),$
$\sigma'(\bar\xi) = \sigma(\bar\xi)$ for all $\bar\xi\in\bar\A\,\backslash\{\sigma^{-1}(\beta\en), \alpha\en, \beta\en\},$

\item[$\Pi_{\alpha\beta}^{\mathrm{le}}$:]
if $\sigma(\alpha\be) = \beta\be,$ then $\sigma' = \sigma;$
otherwise $\sigma'(\beta\en) = \sigma(\alpha\be),$ $\sigma'(\alpha\be) = \beta\be,$\\ $\sigma'(\sigma^{-1}(\beta\be)) = \alpha\be,$ $\sigma'(\bar\xi) = \sigma(\bar\xi)$ for all $\bar\xi\in\bar\A\,\backslash\{\beta\en, \alpha\be, \sigma^{-1}(\beta\be)\}.$
\end{description}

\begin{statement}
\label{prop:steps_bijective_on_schemes}
The elementary induction step $\Pi_{\alpha\beta}^{\mathrm{rb}}$ (the steps $\Pi_{\alpha\beta}^{\mathrm{re}},$ $\Pi_{\alpha\beta}^{\mathrm{lb}},$ $\Pi_{\alpha\beta}^{\mathrm{le}}$) acts on IRE schemes as a bijection from the set of all $\sigma$ such that $\sigma(\alpha\be) = \beta\en$ (such that $\sigma(\alpha\be) = \beta\en,$ $\sigma(\beta\en) = \alpha\be,$ $\sigma(\beta\en) = \alpha\be${),} onto the set of all $\sigma'$ such that $\sigma'(\beta\en) = \alpha\en$ (such that $\sigma'(\beta\be) = \alpha\be,$ $\sigma'(\alpha\en) = \beta\en,$ $\sigma'(\alpha\be) = \beta\be${).}

Each of these steps preserves the number of cycles and the set of irreducible components of a scheme.
\end{statement}

\begin{proof}
We carry out the proof for $\Pi_{\alpha\beta}^{\mathrm{rb}}$; for three other steps it is similar.

So, let us analyze the first block of formulas presented above. Assuming that $\sigma(\alpha\be)  = \beta\en,$ it is easy to see that in both cases they imply $\sigma'(\beta\en) = \alpha\en.$ It is obvious that the two sets of permutations $\{\sigma\colon \sigma(\alpha\be) = \beta\en\}$ and $\{\sigma'\colon \sigma'(\beta\en) = \alpha\en\}$ contain the same number of elements (namly, $(2d-1)!$ elements), hence it is enough to prove that $\Pi_{\alpha\beta}^{\mathrm{rb}}$ maps the first set \emph{onto} the second, i.e., that it is a surjection. Take arbitrary $\sigma'$ such that $\sigma'(\beta\en) = \alpha\en,$ and determine corresponding $\sigma$ in the following way:
\begin{gather}
\label{eq:Pi^rb_alpha_beta^-1_on_schemes}
\text{if}
\quad
\sigma'(\alpha\be) = \beta\en,
\quad
\text{then}
\quad
\sigma = \sigma';
\notag \\
\text{otherwise}
\quad
\sigma(\beta\en) = \sigma'(\alpha\be),
\quad
\sigma(\alpha\be) = \beta\en,
\quad
\sigma((\sigma')^{-1}(\beta\en)) = \alpha\en,
\\
\sigma(\bar\xi) = \sigma'(\bar\xi)
\quad \text{for all}\quad 
\bar\xi\in\bar\A\,\backslash\{\beta\en, \alpha\be, (\sigma')^{-1}(\beta\en)\}.
\notag 
\end{gather}
It is not hard to check that $\sigma(\alpha\be) = \beta\en$ in both cases, and $\Pi_{\alpha\beta}^{\mathrm{rb}}(\sigma) = \sigma',$ therefore, $\sigma'$ indeed has a preimage. (Let us remark that we have just written the explicit formulas for the inverse operator $(\Pi_{\alpha\beta}^{\mathrm{rb}})^{-1}$ acting on schemes, which we will make use of in the next section.)

Concerning the cycles in $\sigma$ and $\sigma'$: a single element $\beta\en$ is first removed from the cycle that contains $\alpha\be,$ and then inserted into the cycle that contatins $\alpha\en.$ These two cycles may coincide, however none of them is or becomes empty, therefore the total number of cycles does not change.

Finally, concerning the irreducible components: those two cycles we mentioned above (which may coincide) belong to the same irreducible component (namely, to the one containing the symbol $\alpha$), and therefore moving the element $\beta\en$ does not change any of the irreducible components.

Proposition~\ref{prop:steps_bijective_on_schemes} is proved.
\end{proof}

Next, we describe the action of the elementary steps on real data. Putting it in words, under the action of $\Pi_{\alpha\beta}^{\mathrm{rb}}$ $(\Pi_{\alpha\beta}^{\mathrm{re}},$ $\Pi_{\alpha\beta}^{\mathrm{lb}},$ $\Pi_{\alpha\beta}^{\mathrm{le}}),$ the two intervals with the label $\alpha$ $(\beta,$ $\alpha,$ $\beta)$ are being cropped on the right (on the right, on the left, on the left) by the length of the two intervals with the label $\beta$ $(\alpha,$ $\beta,$ $\alpha)$, and the interval $I_{\beta\en}$ $(I_{\alpha\be},$ $I_{\beta\en},$ $I_{\alpha\be})$ is moved into the position straight to the right (to the right, to the left, to the left) of the newly cropped interval $I_{\alpha\en}$ $(I_{\beta\be},$ $I_{\alpha\en},$ $I_{\beta\be})$.

Formaly, let $(\sigma', \xect')$ stands for an IRE, obtained from the IRE $(\sigma, \xect)$ as a result of applying an induction step, and the vector $\vect$ is calculated by formulas (\ref{eq:lengths_relation}). Our four induction steps act on vectors of endpoints in the following way:
\begin{description}
\item[$\Pi_{\alpha\beta}^{\mathrm{rb}}:$]
$x'_{\beta\en} = x_{\alpha\en},$ $x'_{\alpha\en} = x_{\alpha\en}-v_\beta,$ $x'_{\bar\xi} = x_{\bar\xi}$\, for all\, $\bar\xi\in\bar\A\,\backslash\{\beta\en, \alpha\en\},$

\item[$\Pi_{\alpha\beta}^{\mathrm{re}}$:]
$x'_{\beta\en} = x_{\alpha\be},$ $x'_{\alpha\be} = x_{\sigma(\beta\be)}-v_\alpha,$ $x'_{\bar\xi} = x_{\bar\xi}$\, for all\, $\bar\xi\in\bar\A\,\backslash\{\beta\en, \alpha\be\},$

\item[$\Pi_{\alpha\beta}^{\mathrm{lb}}$:]
$x'_{\alpha\be} = x_{\beta\en},$ $x'_{\beta\en} = x_{\sigma(\alpha\en)} + v_\beta,$ $x'_{\bar\xi} = x_{\bar\xi}$\, for all\, $\bar\xi\in\bar\A\,\backslash\{\alpha\be, \beta\en\},$

\item[$\Pi_{\alpha\beta}^{\mathrm{le}}$:]
$x'_{\alpha\be} = x_{\beta\be},$ $x'_{\beta\be} = x_{\beta\be} + v_\alpha,$ $x'_{\bar\xi} = x_{\bar\xi}$\, for all\, $\bar\xi\in\bar\A\,\backslash\{\alpha\be, \beta\be\}.$
\end{description}

Let $\vect'$ be the vector of lengths for the IRE $(\sigma', \xect').$ The elementary steps act on vectors of lengths as follows:
\begin{description}
\item[$\Pi_{\alpha\beta}^{\mathrm{rb}}$ or $\Pi_{\alpha\beta}^{\mathrm{lb}}$:]
$v'_\alpha = v_\alpha-v_\beta,$ $v'_\xi = v_\xi$\, for all\, $\xi\in\A\backslash\{\alpha\},$

\item[$\Pi_{\alpha\beta}^{\mathrm{re}}$ or $\Pi_{\alpha\beta}^{\mathrm{le}}$:]
$v'_\beta = v_\beta-v_\alpha,$ $v'_\xi = v_\xi$\, for all\, $\xi\in\A\backslash\{\beta\}.$
\end{description}

\begin{statement}
\label{prop:steps_bijective_on_real_data}
If an elementary induction step transforms a scheme $\sigma$ into a scheme $\sigma',$ then it acts on real components of IREs as a linear bijection between the spaces of allowed endpoints $X_\sigma$ and $X_{\sigma'}$, and as a linear bijection between the spaces of allowed lengths $V_\sigma$ and $V_{\sigma'}.$
\end{statement}

\begin{proof}
Let $\sigma$ be a fixed element from the set of definition of a fixed induction step (i.e., $\sigma(\alpha\be) = \beta\en,$ and the step is $\Pi_{\alpha\beta}^{\mathrm{rb}}$ or $\Pi_{\alpha\beta}^{\mathrm{re}};$ or $\sigma(\beta\en) = \alpha\be,$ and the step is $\Pi_{\alpha\beta}^{\mathrm{lb}}$ or $\Pi_{\alpha\beta}^{\mathrm{le}},$ with certain $\alpha, \beta\in\A$).

The formulas presented above for $\xect'$ and $\vect'$ are linear indeed. The second part of Proposition~\ref{prop:steps_bijective_on_schemes}, in view of Propositions~\ref{prop:X_dimension} and~\ref{prop:V_dimension}, implies that the spaces of allowed endpoints $X_\sigma$ and $X_{\sigma'}$ have the same dimension; the same claim is true for the spaces of allowed lengths $V_\sigma$ and $V_{\sigma'}.$ Hence, it is enough to prove that the induction step maps $X_\sigma$ onto $X_{\sigma'}$ and $V_\sigma$ onto $V_{\sigma'}$ surjectively.

In the case of the step $\Pi_{\alpha\beta}^{\mathrm{rb}}$, for given $\xect'\in X_{\sigma'}$ (and corresponding $\vect'\in V_{\sigma'},$ given by the equalities $v'_\xi = x'_{\sigma'(\xi\be)}-x'_{\xi\be} = x'_{\xi\en}-x'_{\sigma'(\xi\en)},$ $\xi\in\A$) we determine $\xect$ and $\vect$ as follows:
\begin{gather*}
x_{\alpha\en} = x'_{\beta\en},
\quad
x_{\beta\en} = x'_{\sigma'(\alpha\be)} + v'_\beta,
\qquad
x_{\bar\xi} = x'_{\bar\xi}
\quad\text{for all}\quad
\bar\xi\in\bar\A\,\backslash\{\alpha\en, \beta\en\},
\\
v_\alpha = v'_\alpha + v'_\beta,
\qquad
v_\xi = v'_\xi
\quad\text{for all}\quad
\xi\in\A\backslash\{\alpha\}.
\end{gather*}

It is not hard to check that $\xect\in X_\sigma,$ $\vect\in V_\sigma$ relates to $\xect$ in accordance with (\ref{eq:lengths_relation}), and $\Pi_{\alpha\beta}^{\mathrm{rb}}$ indeed transforms $\xect$ into $\xect'$ and $\vect$ into $\vect'.$

For three other elementary induction steps the inverse formulas on real data can be derived and confirmed similarly.

Proposition~\ref{prop:steps_bijective_on_real_data} is proved.
\end{proof}

\begin{remark}
The transformation of lengths $V_\sigma\to V_{\sigma'}$ happens independently on the transformation of endpoints $X_\sigma\to X_{\sigma'},$ therefore it is equally possible to study the induction on floating IREs $(\sigma, \vect)$ or on pinned down IREs $(\sigma, \xect).$ A researcher may choose an approach that is the most appropriate for particular aims.
\end{remark}

\begin{remark}
If it is necessary to guarantee positiveness of IREs during an induction process, then the steps $\Pi_{\alpha\beta}^{\mathrm{rb}}$ and $\Pi_{\alpha\beta}^{\mathrm{lb}}$ can be applied only if $v_\alpha>v_\beta,$ while the steps $\Pi_{\alpha\beta}^{\mathrm{re}}$ and $\Pi_{\alpha\beta}^{\mathrm{le}}$ can be applied only if $v_\alpha<v_\beta.$ It is easy to see that if an $(\sigma, \vect)$ is positive and generic (in the sense that the lengths $v_\xi,$ $\xi\in\A,$ are pairwise different), then for every turn back or forward in a cycle of the permutation $\sigma$ with neighboring elements $\alpha\be$ and $\beta\en,$ $\alpha\ne\beta,$ exactly one of four elementary steps $\Pi_{\alpha\beta}^{\cdot\cdot}$ can be applied with preservation of positiveness.
\end{remark}

\section{Duality and time reversing symmetry}
\label{sect:duality}
Here we define the duality between IRE schemes, which is the cornerstone of our construction, because it allows to invert an induction process in time.

Two IRE schemes $\sigma$ and $\sigma^*$ are called {\em dual} to each other (or {\em mutually dual}), if
\begin{equation}
\label{eq:duality}
\sigma^*(\alpha\be) = \sigma(\alpha\en),
\quad
\sigma^*(\alpha\en) = \sigma(\alpha\be)
\quad
\text{for all}
\qquad
\alpha\in\A.
\end{equation}
The relations (\ref{eq:duality}) determine the {\em duality involution} $\Inv$ in the space of all schemes as $\Inv\sigma = \sigma^*.$ It can be written in the form $\Inv\sigma = \sigma\circ\iota,$ where $\iota$ is a simple reflection of the doubled alphabet $\bar\A$ that swaps $\alpha\be$ with $\alpha\en$ for every $\alpha\in\A.$

It follows straight from the definition that mutually dual schemes have the same irreducible components, however the number of cycles in $\sigma$ and $\sigma^*$ is, in general, different.

Our main result describes the time reversing symmetry for induction on schemes.

\begin{theorem}
\label{theorem:1}
Every one of the following equalities takes place on the set of all those schemes, for which its left-hand side is defined (see Proposition~{\ref{prop:steps_bijective_on_schemes}):}
\begin{gather*}
(\Pi_{\alpha\beta}^{\mathrm{rb}})^{-1} = \Inv\circ\Pi_{\beta\alpha}^{\mathrm{rb}}\circ\Inv,
\\
(\Pi_{\alpha\beta}^{\mathrm{re}})^{-1} = \Inv\circ\Pi_{\alpha\beta}^{\mathrm{lb}}\circ\Inv,
\\
(\Pi_{\alpha\beta}^{\mathrm{lb}})^{-1} = \Inv\circ\Pi_{\alpha\beta}^{\mathrm{re}}\circ\Inv,
\\
(\Pi_{\alpha\beta}^{\mathrm{le}})^{-1} = \Inv\circ\Pi_{\beta\alpha}^{\mathrm{le}}\circ\Inv.
\end{gather*}
\end{theorem}

\begin{proof}
We prove the first equality (the rest are proved similarly).

Accordingly to Proposition~\ref{prop:steps_bijective_on_schemes}, the left-hand side of this equality is defined on the set of schemes $\{\sigma'\colon \sigma'(\beta\en) = \alpha\en\}.$ Consider an arbitrary scheme $\sigma'$ from this set. First of all, $\Inv\sigma'(\beta\be) = \alpha\en,$ which makes $\Pi_{\beta\alpha}^{\mathrm{rb}}$ defined for $\Inv\sigma'.$ The scheme $\sigma = (\Pi_{\alpha\beta}^{\mathrm{rb}})^{-1}\sigma'$ is given by the formulas~(\ref{eq:Pi^rb_alpha_beta^-1_on_schemes}). We need to check that $\Inv\circ\Pi_{\beta\alpha}^{\mathrm{rb}}\circ\Inv\sigma'$ is the same $\sigma.$

If $\sigma'(\alpha\be) = \beta\en,$ then $\sigma = \sigma'.$ On the other hand, $\Inv\sigma'(\alpha\en) = \beta\en,$ therefore $\Pi_{\beta\alpha}^{\mathrm{rb}}\circ\Inv\sigma' = \Inv\sigma',$ and $\Inv\circ\Pi_{\beta\alpha}^{\mathrm{rb}}\circ\Inv\sigma' = \Inv\circ\Inv\sigma' = \sigma' = \sigma.$ The equality holds.

Now consider the case of $\sigma'(\alpha\be)\ne\beta\en.$ Denote $\tau = \Inv\sigma'$ and $\tau' = \Pi_{\beta\alpha}^{\mathrm{rb}}\tau.$ Since $\tau(\alpha\en)\ne\beta\en,$ then $\tau'(\beta\be) = \tau(\alpha\en),$ $\tau'(\alpha\en) = \beta\en,$ $\tau'(\tau^{-1}(\beta\en)) = \alpha\en,$ and $\tau'(\bar\xi) = \tau(\bar\xi)$ for all $\bar\xi\in\bar\A\,\backslash\{\beta\be, \alpha\en, \tau^{-1}(\beta\en)\}.$ Substitution of $\tau = \sigma'\circ\iota$ gives $\tau'(\beta\be) = \sigma'(\alpha\be)$ and $\tau'(\iota((\sigma')^{-1}(\beta\en))) = \alpha\en.$ Hence, we have $\Inv\tau'(\beta\en) = \sigma'(\alpha\be),$ $\Inv\tau'(\alpha\be) = \beta\en,$ $\Inv\tau'((\sigma')^{-1}(\beta\en)) = \alpha\en$, and $\Inv\tau'(\bar\xi) = \sigma'(\bar\xi)$ for all $\bar\xi\in\bar\A\,\backslash\{\beta\en, \alpha\be, (\sigma')^{-1}(\beta\en)\}.$ Thus we arrive at the equality $\Inv\tau' = \sigma.$

Theorem~\ref{theorem:1} is proved.
\end{proof}

At the beginning of Sec.~\ref{sect:dynsys} we defined turns back and forward and the twist number $T = T(\sigma)$ of an IRE scheme $\sigma.$

\begin{statement}
\label{prop:total_turns}
The total number of turns back (as well as turns forward) in two mutually dual IRE schemes is $d.$
\end{statement}

\begin{proof}
Consider an arbitrary IRE scheme $\sigma$ and its dual $\sigma^* = \Inv\sigma.$ For every $\alpha\in\A$, the two-element sets $\{\sigma(\alpha\be), \sigma(\alpha\en)\}$ and $\{\sigma^*(\alpha\be), \sigma^*(\alpha\en)\}$ coincide; denote them by $M_\alpha.$ A turn (back or forward) in a scheme $\sigma$ takes place when either $\sigma(\xi_1\be)  = \xi_2\en$, or $\sigma(\xi_1\en) = \xi_2\be$ for some $\xi_1, \xi_2\in\A,$ and similarly for $\sigma^*.$ The second components of the two elements of $M_\alpha$ can be either both b, or both e, or else one b and one e. It is easy to check that in all three cases the total number of turns at the elements of $M_\alpha$ (back and forward, for $\sigma,$ and for $\sigma^*$) alwais equals~2. Since there are $d$ sets $M_\alpha,$ $\alpha\in\A,$ in total, and $\bigcup_{\alpha\in\A}M_\alpha(\sigma) = \bar\A,$ then the total number of turns (back and forward, for $\sigma,$ and for $\sigma^*$ again) equals~$2d.$ And since the number of turns back equals to the number of turns forward, the statement of the proposition holds.
\end{proof}

We define the {\em twist total} of an IRE scheme $\sigma$ as the value $T(\sigma) + T(\sigma^*),$ i.e., the sum of the twist numbers of two mutually dual schemes $\sigma$ and $\sigma^* = \Inv\sigma.$ Proposition~\ref{prop:total_turns} implies the equality
\begin{equation}
\label{eq:T + T* + N + N* = d}
T(\sigma) + T(\sigma^*) + N(\sigma) + N(\sigma^*) = d,
\end{equation}
where $N(\sigma)$ and $N(\sigma^*)$ are the numbers of cycles in $\sigma$ and $\sigma^*$ respectively. We arrive at the following important observation.

\begin{statement}\label{prop:twist_total}
Elementary induction steps $\Pi_{\alpha\beta}^{\mathrm{rb}},$ $\Pi_{\alpha\beta}^{\mathrm{lb}},$ $\Pi_{\alpha\beta}^{\mathrm{re}},$ and $\Pi_{\alpha\beta}^{\mathrm{le}}$ in their action on IRE schemes preserve their twist total.
\end{statement}

\begin{proof}
Theorem~\ref{theorem:1} implies that if a scheme $\sigma$ is transformed into $\sigma'$ under the action of an induction step, then the scheme $(\sigma')^*$ is transformed into $\sigma^*$ also under the action of an induction step as well. Therefore $N(\sigma') = N(\sigma)$ and $N((\sigma')^*) = N(\sigma^*)$ by Proposition~\ref{prop:steps_bijective_on_schemes}. Now the equality (\ref{eq:T + T* + N + N* = d}) implies $T(\sigma) + T(\sigma^*) = T(\sigma') + T((\sigma')^*).$

Proposition~\ref{prop:twist_total} is proved.
\end{proof}

The last proposition means that whenever an induction step removes a turn from an IRE scheme, this turn does not just disappear, but rather gets transferred into the dual scheme. Let us show how this happens.

Consider the step $\Pi_{\alpha\beta}^{\mathrm{rb}}$ (recall that $\alpha\ne\beta$ by definition); three other cases are similar. It is easy to see that $\Pi_{\alpha\beta}^{\mathrm{rb}}$ removes a turn from the scheme $\sigma$ if and only if $\sigma(\alpha\be) = \beta\en,$ and $\sigma(\beta\en) = \gamma\be$ for some $\gamma\in\A.$ In this case, $\beta\en$ is moved from the position in-between $\alpha\be$ and $\gamma\be$ into the position straight after $\alpha\en,$ one turn back (and one turn forward as well) disappears, and the twist number decreases by~1. Formally, we have $\sigma'(\alpha\be) = \gamma\be$ and $\sigma'(\alpha\en) = \beta\en$ for $\sigma' = \Pi_{\alpha\beta}^{\mathrm{rb}}\sigma.$ In the dual domain, we have $\Inv\sigma(\beta\en) = \alpha\en$ and $\Inv\sigma(\beta\be) = \gamma\be;$
$\Inv\sigma'(\beta\be) = \alpha\en$ and $\Inv\sigma'(\alpha\en) = \gamma\be,$ therefore, as the result of applying the step $(\Pi_{\beta\alpha}^{\mathrm{rb}})^{-1}$ to $\Inv\sigma$, the element $\alpha\en$ is moved from its position straight after $\beta\en$ into the position in-between $\beta\be$ and $\gamma\be,$ one turn back (as well as one turn forward) is being added, and the twist number increases by~1. (One can check that this works for the degenerate cases $\gamma = \alpha$ or $\gamma = \beta$ too.)

\begin{remark}
The last proposition adds one more argument towards the necessity of considering IREs instead of IETs: generically, for an untwisted IET scheme, its dual is twisted and therefore is not an IET scheme, while the twist total characterizes a pair of mutually dual schemes $(\sigma, \Inv\sigma)$ as a whole and does not change due to induction. It is natural to generalize the classical notion of \emph{Rauzy classes} for IETs towards IRE schemes: if one IRE scheme is obtained from another as a result of applying an elementary induction step, then they belong to the same Rauzy class. Now every Rauzy class may contain both untwisted and twisted IRE schemes, however the twist total value is the same for the whole class. Induction can transfer turns between a scheme and its dual, while the twist total remains unchanged.
\end{remark}

\begin{statement}
\label{prop:parity}
The total number of cycles in two mutually dual schemes has the same parity as $d.$ The twist total is always an even number.
\end{statement}

\begin{proof}
It is known that every permutation can be expressed as a product of transpositions and is either even or odd according to the number of transpositions in such a product. Multiplying a permutation by a transposition changes its parity as well as the parity of the number of cycles in it. Every cycle of length $l\ge1$ is a product of $l-1$ transpositions, hence any permutation on $d$ elements is a product of $d-N$ transpositions, where $N\ge1$ is the number of cycles in it. Consider the following surgery procedure on IRE schemes.

Assume for a scheme $\sigma$ that $\sigma(\bar\xi_0)\ne\bar\xi_0$ for certain $\bar\xi_0\in\bar\A.$ Denote $\bar\xi_1 = \sigma^{-1}(\bar\xi_0),$ $\bar\xi_2 = \sigma(\bar\xi_0).$ Now pinch away the element $\bar\xi_0$ from the cycle it belonged to, i.e., transform $\sigma$ into $\sigma'$ such that $\sigma'(\bar\xi_1)  = \bar\xi_2,$ $\sigma'(\bar\xi_0) = \bar\xi_0,$ and $\sigma'(\bar\xi) = \sigma(\bar\xi)$ for all $\bar\xi\not\in\{\bar\xi_0, \bar\xi_1\}.$ It is easy to see that $\sigma' = (\bar\xi_0\leftrightarrow\bar\xi_2)\circ\sigma,$ where $\bar\xi_0\leftrightarrow\bar\xi_2$ denotes the transposition of the elements $\bar\xi_0$ and $\bar\xi_2.$ Since $\bar\xi_2\ne\bar\xi_0$ by our assumption, the parity of the number of cycles in the scheme $\sigma'$ differs from that for $\sigma.$ In the dual domain, we have $\Inv\sigma(\bar\xi_1^*) = \bar\xi_0,$ $\Inv\sigma(\bar\xi_0^*) = \bar\xi_2,$ 
$\Inv\sigma'(\bar\xi_1^*) = \bar\xi_2,$ $\Inv\sigma'(\bar\xi_0^*) = \bar\xi_0,$ and $\Inv\sigma(\bar\xi) = \Inv\sigma'(\bar\xi)$ for all $\bar\xi\not\in\{\bar\xi_0^*, \bar\xi_1^*\}$ (here we denoted $\bar\xi_0^* = \iota(\bar\xi_0)$ and $\bar\xi_1^* = \iota(\bar\xi_1)$). One can see that $\Inv\sigma' = (\bar\xi_0\leftrightarrow\bar\xi_2)\circ\Inv\sigma,$ and hence the parity of the number of cycles in $\Inv\sigma'$ differs from that for $\Inv\sigma.$ It therefore follows that the total number of cycles in the schemes $\sigma$ and $\Inv\sigma$ has the same parity as the total number of cycles in the schemes $\sigma'$ and $\Inv\sigma'.$

So we showed that the surgery described above preserves the parity of the total number of cycles in two mutually dual schemes. Applying it consecutively to every element of $\bar\A$ that the permutation does not send to itself, we will eventually obtain the identity permutation, and dual to the identity is nothing else but $\iota$ (it was defined at the beginning of this section). The number of cycles in the identity is $2d,$ while $\iota$ consists of $d$ cycles. Therefore, the parity of the total number of cycles in two mutually dual schemes is equal to the parity of $d.$

The second statement of this proposition follows from the first one due to the equality (\ref{eq:T + T* + N + N* = d}).

Proposition~\ref{prop:parity} is proved.
\end{proof}

In Sec.~\ref{sect:surfaces} below we show that a positive integer
\begin{gather*}
g = g(\sigma) = \frac{T(\sigma) + T(\Inv\sigma)}{2} + 1
\end{gather*}
is the genus of translation surfaces associated with an IRE with the scheme $\sigma.$

\begin{remark}
A very special class is formed by positive IRE schemes with zero twist total. We call such schemes \emph{rotational} because of their relation to irrational circle rotations. The rotational schemes can be equivalently viewed as (multi-segment) IET schemes such that their duals are IET schemes as well.
\end{remark}

\section{Natural extensions of IREs}
\label{sect:extension}
After we have shown the duality and time reverse symmetry on IRE schemes, the next logical step is to define an IER's \emph{natural extension} as a pair of two IERs $[(\sigma, \xect), (\sigma^*, \yect)]$ (or, in shorter notation, as a triple $(\sigma, \xect, \yect)$), where $\sigma^* = \Inv\sigma,$ $\xect\in X_\sigma,$ and $\yect\in X_{\sigma^*},$ with the following induction dynamics:
\begin{gather*}
\Pi_{\alpha\beta}^{\mathrm{rb}}\big[(\sigma, \xect), (\sigma^*, \yect)\big] = \big[\Pi_{\alpha\beta}^{\mathrm{rb}}(\sigma, \xect), (\Pi_{\beta\alpha}^{\mathrm{rb}})^{-1}(\sigma^*, \yect)\big],
\\
\Pi_{\alpha\beta}^{\mathrm{re}}\big[(\sigma, \xect), (\sigma^*, \yect)\big] = \big[\Pi_{\alpha\beta}^{\mathrm{re}}(\sigma, \xect), (\Pi_{\alpha\beta}^{\mathrm{lb}})^{-1}(\sigma^*, \yect)\big],
\\
\Pi_{\alpha\beta}^{\mathrm{lb}}\big[(\sigma, \xect), (\sigma^*, \yect)\big] = \big[\Pi_{\alpha\beta}^{\mathrm{lb}}(\sigma, \xect), (\Pi_{\alpha\beta}^{\mathrm{re}})^{-1}(\sigma^*, \yect)\big],
\\
\Pi_{\alpha\beta}^{\mathrm{le}}\big[(\sigma, \xect), (\sigma^*, \yect)\big] = \big[\Pi_{\alpha\beta}^{\mathrm{le}}(\sigma, \xect), (\Pi_{\beta\alpha}^{\mathrm{le}})^{-1}(\sigma^*, \yect)\big].
\end{gather*}
In view of Theorem~\ref{theorem:1}, this construction is well-defined 

Since the induction steps transform lengths independently of endpoints, we can also consider \emph{floating natural extensions} of IREs as pairs $[(\sigma, \vect), (\sigma^*, \wect)]$ (or, in shorter notation, as triples $(\sigma, \vect, \wect)$), where $\sigma^* = \Inv\sigma,$ $\vect\in V_\sigma,$ $\wect\in V_{\sigma^*}.$ The lengths are derived from the endpoints by the formulas (\ref{eq:lengths_relation}) and $w_\alpha = y_{\sigma(\alpha_\be)}-y_{\alpha_\be} = y_{\alpha_\en}-y_{\sigma(\alpha_\en)},$ $\alpha\in\A.$
The dynamics of induction are clear:
\begin{gather*}
\Pi_{\alpha\beta}^{\mathrm{rb}}\big[(\sigma, \vect), (\sigma^*, \wect)\big] = \big[\Pi_{\alpha\beta}^{\mathrm{rb}}(\sigma,\vect), (\Pi_{\beta\alpha}^{\mathrm{rb}})^{-1}(\sigma^*, \wect)\big],
\\
\Pi_{\alpha\beta}^{\mathrm{re}}\big[(\sigma, \vect), (\sigma^*, \wect)\big] = \big[\Pi_{\alpha\beta}^{\mathrm{re}}(\sigma, \vect), (\Pi_{\alpha\beta}^{\mathrm{lb}})^{-1}(\sigma^*, \wect)\big],
\\
\Pi_{\alpha\beta}^{\mathrm{lb}}\big[(\sigma, \vect), (\sigma^*, \wect)\big] = \big[\Pi_{\alpha\beta}^{\mathrm{lb}}(\sigma, \vect), (\Pi_{\alpha\beta}^{\mathrm{re}})^{-1}(\sigma^*, \wect)\big],
\\
\Pi_{\alpha\beta}^{\mathrm{le}}\big[(\sigma, \vect), (\sigma^*, \wect)\big] = \big[\Pi_{\alpha\beta}^{\mathrm{le}}(\sigma, \vect), (\Pi_{\beta\alpha}^{\mathrm{le}})^{-1}(\sigma^*, \wect)\big].
\end{gather*}

The \emph{area} of an IRE's natural extension is the real number
\begin{gather*}
S = \vect\cdot\wect = \sum_{\alpha\in\A}v_\alpha w_\alpha
\end{gather*}
that does not change under induction.

\begin{statement}\label{prop:area}
The elementary induction steps acting on IREs' natural extensions preserve their area.
\end{statement}

\begin{proof}
We carry out the proof for $\Pi_{\alpha\beta}^{\mathrm{lb}}$ (for a change); for three other steps it is similar.

Consider an arbitrary permutation $\sigma$ that satisfies $\sigma(\alpha\mathrm{b}) = \beta\mathrm{e},$ $\alpha\ne\beta$ (this is the condition of applicability of such induction step), and arbitrary vectors $\vect\in V_\sigma,$ and $\wect\in V_{\Inv\sigma}.$ Denote $(\sigma', \vect', \wect')  = \Pi_{\alpha\beta}^{\mathrm{lb}}(\sigma, \vect, \wect).$ As we already know from Sec.~\ref{sect:induction}, $\vect' = \Pi_{\alpha\beta}^{\mathrm{lb}}\vect$ means that $v'_\alpha = v_\alpha-v_\beta,$ and $v'_\xi = v_\xi$ for all $\xi\in\A\backslash\{\alpha\}.$ On another hand, $\Inv\sigma = \Pi_{\alpha\beta}^{\mathrm{re}}\Inv\sigma'$ by Theorem~\ref{theorem:1}, $(\Pi_{\alpha\beta}^{\mathrm{re}})^{-1}(\Inv\sigma, \wect) = (\Inv\sigma', \wect')$ by the definition of induction on natural extensions, and $\wect = \Pi_{\alpha\beta}^{\mathrm{re}}\wect'$ means that $w_\beta = w'_\beta-w'_\alpha,$ and $w_\xi = w'_\xi$ for all $\xi\in\A\backslash\{\beta\},$ and therefore, $w'_\beta = w_\beta + w_\alpha,$ $w'_\xi = w_\xi$ for all $\xi\in\A\backslash\{\beta\}.$ Hence, for the area we have
\begin{gather*}
S' 
= \vect'\cdot\wect' 
= \sum_{\xi\not\in\{\alpha, \beta\}}v'_\xi w'_\xi + v'_\alpha w'_\alpha + v'_\beta w'_\beta =
\\
= \sum_{\xi\not\in\{\alpha, \beta\}}v_\xi w_\xi + (v_\alpha-v_\beta)w_\alpha + v_\beta(w_\beta + w_\alpha) 
= \sum_{\xi\in\A}v_\xi w_\xi = S.
\end{gather*}

Proposition~\ref{prop:area} is proved.
\end{proof}

\section{Surfaces and flows associated with IREs}
\label{sect:surfaces}
M.~Veech in~\cite{Veech82} introduced a construction of a surface and a flow suspended over a (classical) IET This construction, known as ``zippered rectangles'', was afterwords routinely used without changes by many (see a well illustrated review of this theory in M.~Viana's paper~\cite{Viana06}, which is freely available on its author's homepage on the internet). We present our generalization of this classical construction and explain why it is even more natural than the original one.

Assume that we have a \emph{positive} extended IRE $(\sigma, \xect, \yect),$ i.e., both $(\sigma, \xect)$ and $(\sigma^*, \yect)$ are positive, where $\sigma^* = \Inv\sigma.$ Consider a disjoint set of $d$ rectangles $R_\xi  = [0, v_\xi]\times[0, w_\xi],$ $\xi\in\A.$ Identify their bottom an upper sides with the beginning and ending intervals $I_{\xi\be}$ and $I_{\xi\en},$ $\xi\in\A,$ of the IRE $(\sigma, \xect)$ respectively, and identify their left and right sides with the beginning and ending segments $I^*_{\xi\be}$ and $I^*_{\xi\en},$ $\xi\in\A,$ of the IRE $(\sigma^*, \yect)$ respectively.

In Sec.~\ref{sect:dynsys}, we described an algorithm transforming a positive IRE into a dynamical system on a tree, which included an arbitrary choice of $T(\sigma)$ number of branching points an certain subsegments. Apply this algorithm first to $(\sigma, \xect),$ and then to $(\sigma^*, \yect)$ (which includes choosing branching points in a number equal to the twist total of $\sigma$).

Now, glue the sides of the rectangles $R_\xi,$ $\xi\in\A,$ in accordance to how the segments in two IREs were glued. I.e., if a beginning point $x\in I_{\xi\be}$ is glued to an ending point $x'\in I_{\xi'\en}$ for some $\xi, \xi'\in\A,$ then the point $(x-x_{\xi\be}, 0)$ of the lower side of the rectangle $R_\xi$ gets glued to the point $(x'-x_{\sigma(\xi'\en)}, w_{\xi'})$ of the upper side of the rectangle $R_{\xi'};$ if a beginning point $y\in I^*_{\eta\be}$ is glued to an ending point of $y'\in I_{\eta'\en}$ for some $\eta, \eta'\in\A,$ then the point $(0, y-y_{\eta\be})$ of the left side of the rectangle $R_\eta$ gets glued to the point $(v_{\eta'}, y'-y_{\sigma^*(\eta'\en)})$ of the right side of the rectangle $R_{\eta'}.$

As a result of such operation, the interiors of the rectangles stay disjoint, while every point of their left sides is identified with some point of their right sides and vice versa, and every point of their lower sides is identified with some point of their upper sides and vice versa. Even more than this, all the sides of all the rectangles can be split into a final number of fragments (subsegments), each of whom is glued to an appropriate opposite side as a whole, i.e., by a parallel translation in the coordinate plain.

The topological surface that we obtain is a two-dimensional oriented manifold without boundary. Geometrically, it is a flat Riemann surface with $T(\sigma) + T(\sigma^*)$ singular points with complete angle $4\pi$ at each one. (This is for a generic case, however nothing prevents singular points from coinciding, and if some $k$ of them do coincide, then they merge into a single singular point with complete angle $2(k + 1)\pi$.) The genus  $g$ of this surface obviously satisfies
\begin{gather*}
2g-2 = T(\sigma) + T(\Inv\sigma),
\end{gather*}
as it was mentioned at the end of Sec.~\ref{sect:duality}.
(A reader can find in~\cite{Veech82} or~\cite{Viana06} more detailed explanation of properties of such surfaces in terms of Riemann manifolds and holomorphic 1-forms.)

On the constructed surface, consider two transversal flows: a vertical flow with a velocity vector parallel to $\overrightarrow{(0, 1)},$ and a horizontal flow with a velocity vector parallel to $\overrightarrow{(1, 0)}$ (both flows naturally branch at the singular points). Both flows can be taken with constant velocity or not, depending on the aims of research. By our construction, the horizontal sides of the rectangles are glued into the horizontal tree, and the dynamical system associated with the IRE $(\sigma, \xect)$ is exactly the Poincare section for the vertical flow described above across this transversal horizontal tree. Similarly, the vertical sides of the rectangles are glued into the vertical tree, and the dynamical system associated with the dual IRE $(\sigma^*, \yect)$ is the Poincare section for the horizontal flow across the transversalvertical tree. The vertical flow can be seen as a suspension flow over the interval exchange on the horizontal tree associated with the IRE $(\sigma, \xect)$, and the horizontal flow can be seen as a suspention flow over the interval exchange on the vertical tree associated with the dual IRE $(\sigma^*, \yect)$. 

Since a classical single-segment IET is a partial case of an IRE, our algorithm can be applied towards a classical IET as well. It is interesting to compare results of our algorithm with classical construction of Veech. Such comparison shows that the permutation (2.1) defined in~\cite{Veech82} describes in fact the dual permutation in our terms, and the vector of height $h$ in~\cite{Veech82} corresponds to our vector of lengths $\wect$ for the dual IRE. However, there is a major difference regarding specific restrictions imposed by Veech to his construction with no explanation whatsoever. In our terms: it is required that all branching points in every cycle must have the same coordinate, i.e., be glued into a single point.
(There is also a second restriction, but it is not that important: in the Veech's construction, one of the branching points must have coordinate $x_{\alpha\be},$ where $\alpha$ is the label of the leftmost beginning interval of the IET.) This unnatural restriction results, on the one hand, in that the number and multiplicity of singular points on the zippered rectangles surface is uniquely determined by the discrete component of the IET. On the other hand, such a restriction imposes corresponding restrictions on the vector of height: in our construction it can be any positive vector allowed by the dual scheme, while for Veech this vector must belong to certain narrowed cone in $\R_ + ^d$ (described in Sec.~3 of the paper~\cite{Veech82}), otherwise the rectangles cannot be zippered the way the author wants.

\subsection{Example}
\label{subsect:example}
For better understanding of the differences between our construction and the one by Veech, consider the most classical example: a single-segment IET that rearranges four itervals in opposite order, i.e. as
\begin{gather*}
\begin{pmatrix}
\alpha & \beta & \gamma & \delta
\\[\medskipamount]
\delta & \gamma & \beta & \alpha
\end{pmatrix}
\end{gather*}
in two-row notation. In terms of an IRE, its scheme $\sigma$ consists of a single cycle
\begin{gather*}
(\alpha\be, \beta\be, \gamma\be, \delta\be, \alpha\en, \beta\en, \gamma\en, \delta\en).
\end{gather*}
The dual scheme $\sigma^* = \Inv\sigma = \sigma\circ\iota$ also consists of a single cycle
\begin{gather*}
(\alpha\be, \beta\en, \gamma\be, \delta\en, \alpha\en, \beta\be, \gamma\en, \delta\be).
\end{gather*}
The latter cycle contains three turns back, therefore the dual scheme's number of twists $T(\sigma^*) = 2,$ and since the scheme $\sigma$ is untwisted, the twist total $T(\sigma) + T(\sigma^*)$ is~2 as well. Hence, after gluing of rectangles by our procedure, the surface obtained will have genus~2 and contain either two singular points with complete angles $4\pi$ at each one, or a single singular point with complete angle $6\pi.$

By Proposition~\ref{prop:V_dimension}, for both schemes $\sigma$ and $\sigma^*$ all vectors of lengths are allowed, so let $\vect$ and $\wect$ be arbitrary positive vectors. An arbitrary choice of coordinates $x_{\alpha\be}$ and $y_{\alpha\be}$ uniquely determines positive IREs $(\sigma, \xect)$ and $(\sigma^*, \yect)$ with the vectors of lengths $\vect$ and $\wect.$ (Alternatively, one may consider floating IREs $(\sigma, \vect)$ and $(\sigma^*, \wect)$ without fixed coordinates, and with fixed lengths only.) Glue these two IREs into trees according to the algorithm from Sec.~\ref{subsect:algorithm}.

For the untwisted IRE $(\sigma, \xect)$, we simply glue each point of every beginning interval to the corresponding unique point of an ending segment with the same coordinate. In this case, the tree is a segment.

For the twisted dual IRE $(\sigma^*, \yect)$ we start by considering the closed polygonal chain with vertices at the points of turns: $(y_{\delta\be}, y_{\beta\en}, y_{\gamma\be}, y_{\delta\en}, y_{\beta\be}, y_{\gamma\en}, y_{\delta\be}).$ It consists of six consecutive segments of lengths $w_\alpha + w_\delta,$ $w_\beta,$ $w_\gamma,$ $w_\alpha + w_\delta,$ $w_\beta,$ $w_\gamma.$ For definiteness, assume that $w_\beta$ is the shortest one of these lengths (in other cases, the algorithm acts similarly, with respective changes). Choose an arbitrary point $c_1\in[y_{\gamma\be}, y_{\beta\en}]$ (this is the shortest segment; there are two of them in fact, both of length $w_\beta,$ and it does not matter, with which one we start). Glue each point of the interval $I^*_{\beta\en}$ lying to the left of $c_1$ to the point of the interval $I^*_{\gamma\be}$ with the same coordinate, and glue each point of the interval $I^*_{\beta\en}$ lying to the right of $c_1$ to the point of the intervals $I^*_{\delta\be}$ or $I^*_{\alpha\be}$ with the same coordinate. Having excluded the glued points from our consideration, we are left with the four-segment closed polygonal chain $(y_{\delta\be}, y_{\delta\en}, y_{\beta\be}, y_{\gamma\en}, y_{\delta\be})$. Choose an arbitrary point $c_2\in[y_{\beta\be}, y_{\gamma\en}]$ (this is the second one of the two shortest segments of length $w_\beta$ in the polygonal chain we started with). Glue each point of the interval $I^*_{\beta\be}$ lying to the left of $c_2$ to the point of the intervals $I^*_{\delta\en}$ or $I^*_{\alpha\en}$ with the same coordinate, adn glue each point of the interval $I^*_{\beta\be}$ lying to the left of $c_2$ to the point of the interval $I^*_{\gamma\en}$ with the same coordinate. The points still unglued form the two-segment closed polygonal chain $(y_{\delta\be}, y_{\delta\en}, y_{\delta\be}),$ so we finish the process by gluing each remaining beginning point to the ending point with the same coordinate.

Now, in the case of $c_1 = c_2 = x_{\alpha\be},$ we get exactly the zippered rectangles described by Veech. Of course, this equality requires additional restrictions imposed on $\wect$, namely, $w_\alpha\le w_\beta\le w_\alpha + w_\gamma$ and $w_\delta\le w_\gamma\le w_\delta + w_\beta.$ The resulting surface (visualized on Fig.~25 in the paper~\cite{Viana06}) has a single singular point with complete angle $6\pi$.

If $c_1\ne c_2,$ then our algorithm produces a geometrically different surface of the same genus. The segments of the dual IRE are glued in different ways depending on the actual ordering of their endpoints on the coordinate axis (this ordering is determined by the vector of lengths) and our choice of the points $c_1$ and $c_2$ during the algorithm.

As an illustration, let us consider one of several possible orderings of the points, namely, the case $y_{\delta\be}<y_{\beta\be}<y_{\gamma\be}<y_{\alpha\en}<y_{\alpha\be}<y_{\gamma\en}<y_{\beta\en}<y_{\delta\en}$ (one can easily derive the corresponding set of inequalities for the lengths $w_\alpha,$ $w_\beta,$ $w_\gamma,$ and $w_\delta,$ using the representation (\ref{eq:lengths-to-endpoints}) of endpoints in terms of the lengths), and assume that in the algorithm we choose $c_1\in(y_{\alpha\be}, y_{\gamma\en})$ and $c_2\in(y_{\beta\be}, y_{\gamma\be}).$ Then, step by step of the procedure, we make the following identifications.

On the first step of the algorithm, we choose $c_1\in(y_{\alpha\be}, y_{\gamma\en})\subset[y_{\gamma\be}, y_{\beta\en}].$ Then, over the segment $[c_1, y_{\beta\en}]$, the ending interval $I^*_{\beta\en}$ is glued to the beginning interval $I^*_{\alpha\be},$ while over $[y_{\gamma\be}, c_1]$, the interval $I^*_{\beta\en}$ is glued to $I^*_{\gamma\be}.$

On the second step, we choose $c_2\in(y_{\beta\be}, y_{\gamma\be})\subset[y_{\beta\be}, y_{\gamma\en}].$ Then, over the segment $[c_2, y_{\gamma\en}]$, the interval $I^*_{\beta\be}$ is glued to $I^*_{\gamma\en},$ while over $[y_{\beta\be}, c_2]$, the interval $I^*_{\beta\be}$ is glued to $I^*_{\alpha\en}.$

At the final step, the rest of identifications is made:
\begin{enumerate}
\item[]
over $[y_{\delta\be}, c_2]$, the interval $I^*_{\delta\be}$ is glued to $I^*_{\gamma\en},$

\item[]
over $[c_2, y_{\alpha\en}]$, the interval $I^*_{\delta\be}$ is glued to $I^*_{\alpha\en},$

\item[]
over $[y_{\alpha\en}, y_{\alpha\be}]$, the interval $I^*_{\delta\be}$ is glued to $I^*_{\delta\en},$

\item[]
over $[y_{\alpha\be}, c_1]$, the interval $I^*_{\alpha\be}$ is glued to $I^*_{\delta\en},$ 

\item[]
over $[c_1, y_{\delta\en}]$, the interval $I^*_{\gamma\be}$ is glued to $I^*_{\delta\en}.$
\end{enumerate}

At the branching point $c_1$, the four intervals $I^*_{\alpha\be},$ $I^*_{\gamma\be},$ $I^*_{\beta\en},$ and $I^*_{\delta\en}$ meet together, while at the branching point $c_2$, the four intervals $I^*_{\beta\be},$ $I^*_{\delta\be},$ $I^*_{\gamma\en},$ and $I^*_{\alpha\en}$ meet.

Finally, we can move to the rectangles $R_\xi,$ $\xi\in\{\alpha, \beta, \gamma, \delta\},$ and describe how their vertical sides, which are associated with the corresponding beginning intervals $I^*_{\xi\be}$ and ending intervals $I^*_{\xi\en}$ of the dual IRE $(\sigma^*, \yect),$ are glued to each other.

The left side of the rectangle $R_\alpha,$ which is the segment $\{0\}\times[0, w_\alpha],$ is split into two fragments:
the lower fragment $\{0\}\times[0, c_1-y_{\alpha\be}]$ is glued to the fragment $\{v_\delta\}\times[y_{\alpha\be}-y_{\alpha\en}, c_1-y_{\alpha\en}]$ of the right side of the rectangle~$R_\delta,$
%
while the upper fragment $\{0\}\times[c_1-y_{\alpha\be}, w_\alpha]$ is glued to the fragment $\{v_\beta\}\times[c_1-y_{\gamma\be}, w_\beta]$ of the right side of $R_\beta.$

The left side of $R_\beta,$ which is the segment $\{0\}\times[0, w_\beta],$ is split into two fragments:
the lower fragment $\{0\}\times[0, c_2-y_{\beta\be}]$ is glued to the fragment $\{v_\alpha\}\times[0, c_2-y_{\beta\be}]$ of the right side of~$R_\alpha,$
%
while the upper fragment $\{0\}\times[c_2-y_{\beta\be}, w_\beta]$ is glued to the fragment $\{v_\gamma\}\times[c_2-y_{\delta\be}, w_\gamma]$ of the right side of~$R_\gamma.$

The left side of $R_\gamma,$ which is the segment $\{0\}\times[0, w_\gamma],$ is split into two fragments:
the lower fragment $\{0\}\times[0, c_1-y_{\gamma\be}]$ is glued to the fragment $\{v_\beta\}\times[0, c_1-y_{\gamma\be}]$ of the right side of~$R_\beta,$
%
while the upper fragment $\{0\}\times[c_1-y_{\gamma\be}, w_\gamma]$ is glued to the fragment $\{v_\delta\}\times[c_1-y_{\alpha\en}, w_\delta]$ of the right side of~$R_\delta.$

The left side of $R_\delta,$ which is the segment $\{0\}\times[0, w_\delta],$ is split into three fragments:
the lower fragment $\{0\}\times[0, c_2-y_{\delta\be}]$ is glued to the fragment $\{v_\gamma\}\times[0, c_2-y_{\delta\be}]$ of the right side of~$R_\gamma,$ 
%
the middle fragment $\{0\}\times[c_2-y_{\delta\be}, y_{\alpha\en}-y_{\delta\be}]$ is glued to the fragment $\{v_\alpha\}\times[c_2-y_{\beta\be}, w_\alpha]$ of the right side of~$R_\alpha,$
%
while the upper fragment $\{0\}\times[y_{\alpha\en}-y_{\delta\be}, w_\delta]$ is glued to the fragment $\{v_\delta\}\times[0, y_{\alpha\be}-y_{\alpha\en}]$ of the right side of $R_\delta.$

It is easy to check that the nine above-listed fragments of the right sides of rectangles cover those sides with no overlapping (beyond their endpoints) or gaps between them.

After the upper sides of the rectangles are glued to the lower sides as well, in accordance with the classical IET under consideration, we obtain a manifold of genus~2, which is, geometrically, a flat surface with two singular points. The first one corresponds to the chosen branching point $c_1$ and is the result of identification of the four points $(0, c_1-y_{\alpha\be})\in R_\alpha,$ $(v_\beta, c_1-y_{\gamma\be})\in R_\beta,$ $(0, c_1-y_{\gamma\be})\in R_\gamma,$ and $(v_\delta, c_1-y_{\alpha\en})\in R_\delta.$ The second singular point corresponds to $c_2$ from our algorithm and results from gluing the four points $(v_\alpha, c_2-y_{\beta\be})\in R_\alpha,$ $(0, c_2-y_{\beta\be})\in R_\beta,$ $(v_\gamma, c_2-y_{\delta\be})\in R_\gamma$, and $(0, c_2-y_{\delta\be})\in R_\delta$ into one as well. The complete angle at each of these two singular pooints is equal to $4\pi.$

In the review~\cite{Viana06} (which is, as we remarked earlier, freely available on the internet), one can find an illustrated description of Veech's construction for the same example and thus visualize the difference between two algorithms.

\section{Conclusions}
\label{sect:conclusions}
In the present paper, we introduced the new concept of an interval rearrangement ensemble (IRE), which is a generalization of the classical construction of an interval exchange transformation (IET). Unlike the latter one, an IRE is not, generally speaking, a dynamical system, although it can be transformed into one by gluing its intervals into a tree, with several degrees of freedom in choice of its parameters.

Our construction is very natural and was overlooked by other researches maybe because of they restricted themselves to studying single-segment IETs. It was almost realized by Veech in~\cite{Veech82} in investigating vertical sides of ``zippered rectangles'', however their configuration was not recognized as dual to the initial IET on horizontal sides.

An induction of Rauzy\,--\,Veech type can be applied to IREs (and a renormalization can be defined based of this induction). Elementary induction steps on schemes (which is our term for the discrete components of IREs) are conjugate with inverse steps by a simple involution that defines the duality between schemes. A natural extension of an IRE can be seen as a pair of IREs with mutually dual schemes, and elementary induction steps on natural extensions are invertible as well and conjugate to inverse steps by an involution that simply swaps the elements of such a pair.

For a given natural extension of an IRE, when we take a set of parallel rectangles with sizes determined by the lengths of intervals in the mutually dual IREs forming this natural extension, glue each of these IREs into a tree and then glue the sides of the rectangles acoordingly, we obtain a flat translation surface with a certain number of singular points. On this surface, two transversal flows of constant directions are determined. Each of these flows is parallel to one of the two trees glued up from the sides of the rectangles and produces a Poincare section on the second tree. This Poincare section is nothing else but the dynamical system the corresponding IRE was glued into---an interval exchange on a tree.

In the special case of a rotational IRE (i.e., a multi-segment IET, which is a first return map for an irrational circle rotation onto a finite union of arcs), the dual IRE is rotational as well, and any corresponding translation surface has genus~$1,$ i.e., it is a torus without singular points.

\end{document}